\documentclass{article}
\usepackage[dvipdfmx]{graphicx}

\usepackage{amsmath}

\usepackage{amsthm}
\usepackage{amssymb}
\usepackage{amsbsy}
\usepackage{amsfonts}
\usepackage{mathrsfs}
\usepackage[all]{xy}
\usepackage{amstext}
\usepackage{amscd}
\usepackage{psfrag}
\usepackage{enumerate}
\usepackage{flafter}
\allowdisplaybreaks

\theoremstyle{plain}
\newtheorem{thm}{Theorem}[section]
\newtheorem{prop}[thm]{Proposition}
\newtheorem{cor}[thm]{Corollary}
\newtheorem{lem}[thm]{Lemma}
\theoremstyle{definition}
\newtheorem{exa}[thm]{Example}

\newtheorem{rem}[thm]{Remark}
\newtheorem{defn}[thm]{Definition}

\newtheorem{fact}[thm]{Fact}

\def\dim{\mathop{\mathrm{dim}}\nolimits}

\def\Ker{\mathop{\mathrm{Ker}}\nolimits}

\def\Ext{\mathop{\mathrm{Ext}}\nolimits}

\newcommand{\GG}{{\cal G}}

\newcommand{\AAA}{{\Omega}}

\newcommand{\SL}{{\rm SL}}
\newcommand{\SO}{{\rm SO}}

\newcommand{\Isom}{{\rm Isom}}

\newcommand{\Tor}{{\rm Tor}}
\newcommand{\SU}{{\rm SU}}

\newcommand{\Ss}{{\cal S}}

\newcommand{\ra}{\rightarrow}
\newcommand{\Q}{{\mathbb{Q}}}
\newcommand{\R}{{\mathbb{R}}}
\newcommand{\Z}{{\mathbb{Z}}}
\newcommand{\N}{{\mathbb{N}}}

\newcommand{\C}{{\mathbb{C}}}

\newcommand{\g}{{\mathfrak g}}
\newcommand{\ww}{{\omega}}

\def\Diff{\mathop{\mathrm{Diff}}\nolimits}
\def\Tr{\mathop{\mathrm{Tr}}\nolimits}

\begin{document}
\large
\begin{center}
{\bf\Large  Chern--Simons-type $3$-cocycles and
$\mathbb{Q}/\mathbb{Z}$-torsion in the third homology
of discrete diffeomorphism groups}
\end{center}
\begin{center}{Takefumi Nosaka\footnote{E-mail address: {\tt nosaka@math.titech.ac.jp}
}}\end{center}
\begin{abstract}\baselineskip=12pt \noindent
We prove that the third integral group homology of both the
volume-preserving diffeomorphism group and the strict
contactomorphism group of the standard \(3\)-sphere contains
\((\mathbb Q/\mathbb Z)^2\).  The proof constructs two independent
locally smooth Chern--Simons-type \(3\)-classes from the volume form
and the standard framing of \(S^3\).  A torsion-controlled local
integration method and comparison with the primitive
Cheeger--Chern--Simons class also detect \(\mathbb Q/\mathbb Z\) for
several spherical space forms and symplectic projective manifolds.
\end{abstract}
\begin{center}

\normalsize

\baselineskip=11pt


\end{center}

\begin{center}

\normalsize

\baselineskip=11pt

{\bf MSC 2020 } \\

\ \ \ \ 57M50, 57S05, 58D05, 22E65, 20J06 \ \ \ \ \

\end{center}


\large
\baselineskip=16pt

\tableofcontents
\section{Introduction}\label{sec1}

The cohomology of diffeomorphism groups lies at the crossroads of topology,
geometry, foliation theory, and algebraic $K$-theory.
Degree-two classes are often closely related to central extensions and characteristic
class theory (see, e.g., \cite{Fuks,FSS,Neeb}).
By contrast, degree-three phenomena for geometrically defined diffeomorphism groups
are less well understood, especially when one seeks locally smooth cocycles that can be
evaluated and compared with classical secondary invariants.
A substantial part of the earlier literature concerns rational,
real, or continuous cohomological invariants.  

The aim of this paper is to construct locally smooth
$\R/\Z$-valued group $3$-cocycles of Chern--Simons type on several diffeomorphism groups
arising in low-dimensional and symplectic topology, to detect their nontriviality,
and thereby to exhibit $\Q/\Z$-torsion in third group homology.

Throughout, $M$ denotes a closed oriented smooth $n$-manifold, and $\Diff(M)$ denotes the group
of diffeomorphisms of $M$ endowed with the Whitney topology.
Given a differential form $\lambda$ on $M$, we denote by $\Diff_\lambda(M)\subset \Diff(M)$
the subgroup preserving $\lambda$.
For a (topological) group $G$, we write $H_*^{\rm gr}(G)$ for the integral group homology of the
underlying discrete group; see Section~\ref{sec2991} for conventions.

We state our three main results.

\begin{thm}[{Sphere case}]\label{main1}
Let $M:=S^3$. Let $v_{S^3}$ be the standard volume 3-form. 
Then $H_3^{\rm gr}\bigl(\Diff_{v_{S^3}}(S^3)\bigr)$ contains a subgroup isomorphic to $\Q/\Z \oplus \Q/\Z$.
\end{thm}

\begin{thm}[{Spherical geometry}]\label{main2}
Let \(\Gamma\subset\SU(2)\) be one of the binary dihedral, tetrahedral,
octahedral, or icosahedral subgroups with $|\Gamma|>8$.  
Regard the quotient $M:=\SU(2)/\Gamma$ as a 3-manifold, and let $v$ denote the $3$-form on $M$ induced by the volume 3-form $v_{S^3}$ on $S^3$.

Then $H_3^{\rm gr}\bigl(\Diff_v(M)\bigr)$ contains a subgroup isomorphic to $\Q/\Z$.
\end{thm}
\begin{thm}[{Symplectic case}]\label{main3}
Let $\C P^n$ be equipped with the symplectic $2$-form $\ww_n$ arising from the Fubini--Study metric.
Then the third group homology with integer coefficients of each of the following groups
contains a subgroup isomorphic to $\Q/\Z$:
\[
\Diff_{\ww_1}(\C P^1),\qquad
\Diff_{\ww_1\oplus \ww_1}(\C P^1\times \C P^1),\qquad
\Diff_{\ww_2}(\C P^2).
\]
\end{thm}

\noindent\textbf{Proof strategy.}
The model for the detection argument is the primitive
Cheeger--Chern--Simons class
\[
 [c_P^{\rm CS}]
 \in H^3_{\rm loc,sm}(\SU(2);\mathbb R/\mathbb Z).
\]
For the standard rotation homomorphisms
\(R_m:\mathbb Z/m\to\SU(2)\), one may choose
direct-limit-compatible generators
\(\tau_m\in H_3^{\rm gr}(\mathbb Z/m;\mathbb Z)\) such that
\[
 \left\langle R_m^*[c_P^{\rm CS}],\tau_m\right\rangle
 =
 [1/m].
\]
The compatible cyclic classes therefore define a homomorphism from
\(\mathbb Q/\mathbb Z\) to
\(H_3^{\rm gr}(\SU(2);\mathbb Z)\), and evaluation against
\([c_P^{\rm CS}]\) makes its composite the standard inclusion
\(\mathbb Q/\mathbb Z\subset\mathbb R/\mathbb Z\).
This is the basic finite-cyclic detection mechanism used throughout
the paper.

We realize the same detection mechanism in the diffeomorphism groups
appearing in the main theorems.  Given a closed left-invariant form
\(\widetilde\lambda\) on a Fr\'echet--Lie group, we integrate it over
face-compatible local simplices indexed by \(U\)-small tuples.  A
convex chart at the identity provides these simplices with smooth
dependence on their vertices; see
Section~\ref{sec223344}.

To pass from this local cochain to a genuine locally smooth group
cocycle, we impose a period condition and control the configured
homology below the top degree.  In the spherical cases, the
volume-form fibration and Smale-type homotopy equivalences reduce the
required homotopy information to the corresponding isometry groups
\cite{EM,HKMR,BK}.  In the symplectic cases, we use the classical
model for \(\mathbb{CP}^{1}\) and the known models for
\(\mathbb{CP}^{2}\) and
\(\mathbb{CP}^{1}\times\mathbb{CP}^{1}\) \cite{Gro,AM}.
These facts allow us to extend the local filling
degree by degree to a torsion-twisted partial filling of the full
homogeneous bar complex by smooth singular chains.  Under the
corresponding spherical period condition, this construction produces
a locally smooth group 3-cocycle whose local restriction is
\(|\pi_1(\Diff_v(M))|^3 \mathcal F_\sigma(\widetilde\lambda)\); see
Section~\ref{secDiff0Configured}.

The geometric input consists of explicit invariant \(3\)-forms.  In
the spherical cases we use suitably normalized pullbacks of the
volume form by evaluation maps.  For the symplectic examples, the
\(L^2\)-pairing on normalized Hamiltonians and the Poisson bracket
give the Cartan \(3\)-cocycle \(\beta_M\); see
Definition~\ref{def:betaM}.  For \(S^3\), we also construct the
framing cocycle \(\eta_{\rm fr}\) of
Lemma~\ref{lem:framing-cocycle-S3}.  The proofs use the periods and
the explicit restrictions of these forms to compact
\(\SU(2)\)-actions; no identification with Euler or Pontryagin
transgressions is required.

For the final detection, we pull the resulting classes back along
explicit smooth homomorphisms from \(\SU(2)\).
Lemma~\ref{lem:comparison-CCS}, proved using naturality of
differentiation and its injectivity in degree three for \(\SU(2)\),
identifies these pullbacks with the corresponding integral multiples
of the primitive Cheeger--Chern--Simons class.  Every nonzero
multiple used in the one-dimensional detection gives a homomorphism
from \(\mathbb Q/\mathbb Z\) with finite kernel, and hence an image
isomorphic to \(\mathbb Q/\mathbb Z\).  In the \(S^3\) case, the left
and right cyclic systems, together with two cocycle evaluations,
give an integer \(2\times2\) pairing matrix of nonzero determinant.
This gives a homomorphism from
\((\mathbb Q/\mathbb Z)^2\) with finite kernel and hence detects a
subgroup isomorphic to \((\mathbb Q/\mathbb Z)^2\).
When the full diffeomorphism group is disconnected, we verify that
the component group acts trivially on the detected subgroup and then
apply finite-index homology transfer.

\noindent\textbf{Further results and relation to earlier work.}
Beyond the three main theorems, we construct a real-valued locally
smooth cocycle in the hyperbolic case, whose nontriviality remains
open, and use differentiation to prove a non-extension result for
the two \(S^3\)-classes.  We also discuss an application to Quillen
plus constructions.

Reznikov introduced
\(\mathbb R/\mathbb Z\)-valued Chern--Simons-type secondary classes
in symplectic topology in~\cite{Rez1}; the later work~\cite{Rez2}
constructs their real continuous-cohomology counterparts for groups
of volume-preserving and symplectic diffeomorphisms.  
Our construction yields locally smooth representatives with explicitly
prescribed identity germs and computable restrictions to compatible finite
cyclic subgroups. The global set-theoretic extensions are not canonical,
and we do not claim equality with Reznikov's regulator classes
In contrast to Reznikov's framework, our approach is independent and provides explicit, locally smooth cocycle representatives whose restrictions to compatible finite cyclic subgroups are directly computable.

\noindent\textbf{Organization of the paper.}
Section~\ref{sec299} reviews group cohomology, configured domains,
and Dupont-type fillings.  Section~\ref{sec33233} constructs local
affine fillings, proves the torsion-controlled extension theorem, and
relates the resulting locally smooth cocycles to Lie algebra
cohomology; the proof of the differentiation formula is given in
Appendix~\ref{furoku}.  Section~\ref{sec2233} recalls the primitive
Cheeger--Chern--Simons class and proves the spherical results,
including the detection of a subgroup isomorphic to
\((\mathbb Q/\mathbb Z)^2\) for \(S^3\) and its strict-contact
consequence; it also contains the non-extension result and the
hyperbolic discussion.  Section~\ref{sec32884} constructs the
symplectic invariant forms, proves Theorem~\ref{main3}, and discusses
the application to Quillen plus constructions.

\noindent\textbf{Conventions.}
Throughout, $\GG$ denotes a topological group with identity element $e$.
Unless stated otherwise, finite-dimensional manifolds are connected, closed, oriented, and smooth.

\subsection*{Acknowledgments}
This work was partially supported by JSPS KAKENHI, Grant Number 00646903.
The author is also grateful to Shuhei Maruyama and Sam Nariman for helpful comments.

AI tools were used only for language editing and preliminary consistency checks. 
All mathematical statements and arguments are the author's responsibility.

\section{Configured domains and Dupont-type fillings}
\label{sec299}
Section~\ref{sec2991} briefly reviews group homology and cohomology
in the homogeneous model, together with smooth and locally smooth
group cohomology. Section~\ref{sec2d991} introduces configured domains, integral and
localized acyclicity of their associated chain complexes, and the
corresponding cochain complexes.
Section~\ref{sec2e992} then recalls Dupont-type fillings and the
configured cochains obtained by integrating invariant differential
forms.

The configured-domain language provides the homological framework,
while Dupont's construction provides the geometric model for the
\(U\)-small domains and local affine fillings used later.

\subsection{Review of group (co)homology and (locally) smooth group cohomology}\label{sec2991}

We first recall group (co)homology using the homogeneous complex.
Throughout this subsection $G$ denotes a (discrete) group.

For $n \ge 0$ and $0 \le i \le n$, we define the $i$-th face map
\begin{equation}\label{maxx}
 d_i^n:G^{n+1}\longrightarrow G^n;\qquad
 (g_0,\dots,g_n)\mapsto (g_0,\dots,g_{i-1},g_{i+1},\dots,g_n).
\end{equation}
Let $C_n^{\Delta}(G  )$ be the free $\Z$-module $\Z \langle G^{n+1} \rangle$ with basis $G^{n+1}$.
We consider the diagonal action $G \curvearrowright G^{n+1}$ and denote by
$ C_n^{\Delta}(G  )_{G}$ the module of coinvariants, i.e., the quotient module of
$C_n^{\Delta}(G  )$ by the 
action.
For $a \in G^{n+1}$, we set
\[
 \partial_n [a] := \sum_{i: 0 \leq i \leq n} (-1)^i [ d_{i}^{n}(a) ] \in C_{n-1}^{\Delta}(G  )_{G},
\]
and extend $\Z$-linearly to obtain a boundary map
$\partial_n:C_n^{\Delta}(G)_G\ra C_{n-1}^{\Delta}(G)_G$.
The homology of this quotient complex is called the \emph{group homology of $G$} and is denoted
by $H_n^{\rm gr}(G )$.

Dually, we recall the definition of group cohomology (see \cite{Bro} for details).
Let $A$ be an abelian group.
We define $ C^n_{\Delta} ( G ;A) := \{ f : G^{n+1} \ra  A \}$
to be the module consisting of all maps $G^{n+1} \to A$.
For $f \in C^n_{\Delta} ( G ;A)$, we define the coboundary
$\partial^n(f): G^{n+2} \to A$ by the usual simplicial formula $
 \partial^n(f)(a)=\sum_{i: 0 \leq i \leq n+1} (-1)^i f \bigl( d_{i}^{n+1}(a) \bigr)$ for $a\in G^{n+2}$.
Again $G$ acts diagonally on $G^{n+1}$, and we denote the $G$-invariant part of
$C^n_{\Delta} ( G ;A)$ by $C^n_{\Delta} ( G ;A)^{G}$.
Since $\partial^{n+1} \circ \partial^n =0$, the \emph{group cohomology}
$H^n_{\rm gr}(G;A)$ is defined as
\[
 H^n_{\rm gr}(G;A)
 :=
 \frac{\Ker(\partial^n ) \cap C^n_{\Delta} ( G ;A)^{G}}
      {\partial^{n-1}\bigl(C^{n-1}_{\Delta} ( G ;A)^{G}\bigr)}.
\]
If $G$ is regarded as a discrete group, then the ordinary cohomology of the
Eilenberg--MacLane space $K(G,1)$ with coefficients in $A$ is canonically
isomorphic to $H^n_{\rm gr}(G; A)$.
We also consider the pairing (evaluation)
\[\langle \, , \, \rangle :  C^n_{\Delta} ( G ;A) ^G \times C_n^{\Delta} ( G  )_G \longrightarrow A; \qquad (f,a) \mapsto f(a). \]
This induces a pairing $H^n_{\rm gr} ( G ;A) \times H_n^{\rm gr} ( G  ) \ra A$, called {\it the Kronecker pairing}.

%

Next, we recall smooth and locally smooth group cohomology.
From now on, we assume that $\GG$ is a (Fr\'echet) Lie group
and that $A$ is one of
$\R$, $\R/\Z$, $\C$, or $\C/\Z$ endowed with the standard topology, regarded as a trivial
topological $\GG$-module.
We denote by $C^n_{\rm sm} ( \GG ;A)^{\GG}$ the group of $\GG$-invariant smooth maps
$\GG^{n+1} \to A$, forming a subcomplex of $C^*_{\Delta} ( \GG ;A)^{\GG}$.
We call the cohomology of globally smooth homogeneous cochains the
global smooth cohomology and denote it by \(H^*_{\rm sm}(\GG;A)\).
The inclusion $ C^*_{\rm sm} ( \GG ;A)^{\GG} \hookrightarrow C^*_{\Delta} ( \GG ;A)^{\GG}$ induces a homomorphism
$\mathrm{Comp}:H^*_{\rm sm}(\GG;A)\ra H^*_{\rm gr}(\GG;A)$.
We call $\mathrm{Comp}$ the \emph{comparison map}.
In the sequel we shall always work with smooth homogeneous cochains
and refer to $H^*_{\rm sm}(\GG;A)$ as the smooth cohomology of $\GG$.


Finally, we recall the notion of locally smooth cohomology (see, for example, \cite{WW,Neeb}).
For $n \ge 0$, we define the group of $n$-cochains
\[
C^n_{\mathrm{loc,sm}}(\GG;A)
:=\left\{\, f:\GG^{n+1}\longrightarrow A \ \middle|\
\begin{aligned}
&\text{There exists an open subset } \,U_f\subset \GG^{n+1} \text{ such that}\\
& (e, \dots, e) \in U_f,\text{ and the restriction } f|_{U_f}\ \text{is smooth.}
\end{aligned}
\right\}.
\]
We denote by $C^*_{{\rm loc,sm}} ( \GG ;A)^{\GG}$ the subcomplex of $\GG$-invariant cochains,
and write $ H^n_{{\rm loc,sm}}(\GG;A)$
for the corresponding cohomology $ H^n\bigl( C^*_{{\rm loc,sm}} ( \GG ;A)^{\GG}, \partial^* \bigr)$.
For a detailed comparison, we refer the reader to \cite{WW,Neeb}.

\subsection{Configured domains and cohomological comparison}
\label{sec2d991}
The full homogeneous complex is often larger than the domain on which a
geometric group cocycle is naturally defined.  We therefore separate three
pieces of structure: a restricted simplicial domain, its homological
comparison with the homogeneous bar complex, and a smooth filling on that
domain.  The first two are introduced in this subsection; fillings are
discussed in the next subsection.

This distinction will also clarify the role of the \(U\)-small domains used
in Section~\ref{sec223344}.  Those domains are chosen primarily because they support canonical
local fillings.  Their extension to the full homogeneous complex will be
obtained by a torsion argument rather than by assuming their acyclicity.
\begin{defn}\label{asfa}
Let \(\GG\) be a Fr\'echet--Lie group.  A sequence of open subsets
\[
 \Ss=\{S_n\subset\GG^{n+1}\}_{n\geq0}
\]
is called a \emph{configured domain} if \(S_0=\GG\), each \(S_n\)
is invariant under the diagonal action of \(\GG\), and
\begin{equation}\label{condition1} 
 d_i^n(S_n)\subset S_{n-1}
 \qquad
 (0\leq i\leq n).
\end{equation} 

Fix integers \(q\geq1\) and \(\nu\geq1\), and let
$ \mathbb Z[1/\nu]\langle S_n\rangle$ be the free $\Z[1 /\nu]$-module generated by $ S_n.$
We say that \(\Ss\) is
\emph{\(\mathbb Z[1/\nu]\)-acyclic through degree \(q-1\)}, or
equivalently \emph{acyclic away from \(\nu\) through degree \(q-1\)},
if the augmented configured chain complex
\[
 \mathbb Z[1/\nu]\langle S_q\rangle
 \longrightarrow
 \mathbb Z[1/\nu]\langle S_{q-1}\rangle
 \longrightarrow\cdots\longrightarrow
 \mathbb Z[1/\nu]\langle S_0\rangle
 \longrightarrow
 \mathbb Z[1/\nu]
 \longrightarrow0
\]
is exact at every displayed term except possibly the leftmost one.

When \(\nu=1\), we say that \(\Ss\) is
\emph{integrally acyclic through degree \(q-1\)}.
\end{defn}
\begin{rem}
Equivalently, \(\Ss\) is acyclic away from \(\nu\) through degree
\(q-1\) if its augmented configured complex is exact through that
degree after localization at every prime \(p\nmid\nu\).  Thus the
condition depends only on the prime divisors of \(\nu\), and is weaker
than requiring \(\nu\) itself to annihilate the homology.
\end{rem}


For an abelian group \(A\) with trivial \(\GG\)-action and \(n\geq0\), define
\[
C^n(\Ss;A)^{\GG}  := \{ f:S_n\longrightarrow A \mid f \text{ is $\GG$-invariant} \}.\] 
Since \(\Ss\) is a configured domain, these groups form a cochain
complex.
By the simplicial formula for $\partial^*$, these form cochain complexes, and 
we
denote the cohomology groups by $H_{\Ss}^*(\GG;A)$. 
Furthermore, by a standard discussion of comparison theorem (see \cite{Bro}), if \(A\) is a \(\mathbb Z[1/\nu]\)-module with trivial \(\GG\)-action, then restriction induces an isomorphism
$
 H^k_{\rm gr}(\GG;A)
 \ra 
 H^k_{\Ss}(\GG;A)
$
for \(k<q\), and an injection for \(k=q\).

We note that the comparison map $\mathrm{Comp}:H_{\rm sm}^*(\GG;A)\to H^*_{\rm gr}(\GG;A)$
admits a comparison zigzag of the form
\begin{equation}\label{azaa}
 \mathrm{Comp}
 :H_{\rm sm}^*(\GG;A) \stackrel{i^*}{\longrightarrow}  H_{\rm \Ss, sm}^*(\GG;A) \longrightarrow
 H_{\Ss }^*(\GG;A) \stackrel{\simeq}{\longleftarrow} H^*_{\rm gr}(\GG;A)
\end{equation}
where $i^*$ is induced by $i_n$, and $* <q$.
Thus, in the integrally acyclic case, configured cochains provide an
intermediate model for the comparison map from globally smooth group
cohomology to ordinary group cohomology.
One of the main practical uses of such cocycles is their pairing with cycles in
$H_*^{\rm gr}(\GG)$.

It is in general difficult to verify the required acyclicity. 
However, we record the following classical sufficient criterion for acyclicity.
\begin{lem}\label{iine}
Let \(\Ss\) be a configured domain, and let $\nu =1$.  Suppose that, for every
\(0\leq n\leq q\) and every finite subset \(F\subset S_n\), there exists
\(y\in\GG\) such that
\((x_0,\ldots,x_n,y)\in S_{n+1}\) for every
\((x_0,\ldots,x_n)\in F\).
Then \(\Ss\) is 
\(\mathbb Z\)-acyclic through degree \(q-1\).
\end{lem}
\begin{proof}
Append \(y\), with sign \((-1)^{n+1}\), to every simplex occurring in a
cycle.  The usual cone identity shows that the resulting chain has the
given cycle as its boundary.  In degree zero, the additional \(y\)-term
vanishes because the cycle has augmentation zero.
\end{proof}

\begin{exa}[Pairwise-distinct configurations]\label{aa}
Let \(\GG\) be an infinite Lie group, and let \(S_n\subset\GG^{n+1}\)
consist of the tuples with pairwise distinct entries.
These sets form an open configured domain.  Given finitely many such
tuples, one may choose a new entry distinct from all entries occurring in
them.  Lemma~\ref{iine} therefore shows that this domain is
\(\mathbb Z\)-acyclic through every finite length.  This is the standard
configuration complex occurring, for example, in algebraic \(K\)-theory
and scissors-congruence constructions; see \cite{Wei,Dup2,DZ}.
\end{exa}

\begin{exa}\label{i125511}
As a related example, consider the case $G=\SL_2(\C)$.
Define a map
\(h:\SL_2(\C)\rightarrow\mathbb{C}P^1\), \(A\mapsto A\cdot[1:0]\),
whose restriction to $\SU(2)$ is the Hopf fibration.
For $n > 0$, we set
\begin{equation}\label{989796}
 S_n^{\neq h}(G)
 :=
 \bigl\{(g_0,g_1,\dots,g_n) \in \SL_2(\C)^{n+1} \mid h(g_i) \neq h(g_j) \text{ for } i \neq  j \bigr\},
\end{equation}
and $ S_0^{\neq h} (G) =G$. 
As in Example~\ref{aa}, the same argument shows that this sequence is 
\(\mathbb Z\)-acyclic through every finite length.
This complex is discussed in \cite{DZ} and reappears in Section~\ref{sec2233}.
\end{exa}
Further examples appear in Theorem~\ref{tsukau}.

Configured domains provide restricted face-stable domains on which
geometric constructions may be defined.  Acyclicity of the associated
configured chain complex is a separate homological condition:
integral acyclicity gives comparison with ordinary group cohomology,
whereas acyclicity away from \(\nu\) gives the corresponding
comparison for coefficients on which \(\nu\) acts invertibly.

\subsection{Fillings and Dupont's classical model}
\label{sec2e992}
The concept of \(\mathbb Z\)-acyclic configured domain is a homological property of a configured domain.
To integrate differential forms, we need additional geometric data: a
face-compatible family of smooth singular simplices.  We call such data a
filling.  

We begin by fixing notation.
Let $K\subset \GG$ be a closed Lie subgroup, and assume that
the quotient $\GG/K$ carries a natural structure of a Fr\'echet manifold.
Let $\Delta^i\subset \R^{i+1}$ denote the standard $i$-simplex.
For each $i\ge0$, we let $C_{i}^{\rm sm} (\GG/K )$ be the free abelian group generated by
all smooth maps $\Delta^i \to \GG/K$.
Then $C_{*}^{\rm sm} (\GG/K )$ is a subcomplex of the usual singular chain complex of
$\GG/K$.
Via the natural left action $\GG\curvearrowright \GG/K$, the complex $C_{*}^{\rm sm} (\GG/K )$ becomes a left $\Z[\GG]$-module in each degree.

\begin{defn}[{cf.~\cite{Dup2,DK}}]\label{33}
Let $\Ss=\{S_n\subset \GG^{n+1}\}_{ 0 \leq n\leq q }$ be a sequence of length $q$ satisfying
the face condition \eqref{condition1}.
We say that a $\Z[\GG]$-module homomorphism
\[
 \sigma_n : \Z \langle S_n \rangle \longrightarrow C_{n}^{\rm sm} (\GG/K ) ,
 \qquad 0\le n\le q,
\]
is a \emph{filling over $\Ss$} (of length $q$) if the family
$\sigma=\{\sigma_n\}_{n\leq q }$ extends to a chain map between augmented complexes, i.e.\ if
the following diagram commutes:
\[
\xymatrix{
 \Z\ar@=[d]
 & \Z \langle S_0 \rangle \ar[l]_{\!\!\!\!\!\! \epsilon } \ar[d]^{ \sigma_0}
 &  \Z \langle S_1 \rangle\ar[l]_{\partial_1} \ar[d]^{ \sigma_1}
 & \ar[l]_{\ \ \ \ \partial_2}  \cdots
 & \ar[l]_{\!\!\!\!\!\!\!\!\!\!\!\!\!\!\! \partial_{q }}  \Z \langle S_q  \rangle \ar[d]^{\sigma_q}
 \\
 \Z
 & C_{0}^{\rm sm} (\GG/K ) \ar[l]_{\!\!\!\!\!\!\!\!\!\!\!\!\epsilon }
 & C_{1}^{\rm sm} (\GG/K )\ar[l]_{\partial_1}
 & \ar[l]_{\ \ \ \ \partial_2}   \cdots
 &  C_{q }^{\rm sm} (\GG/K ). \ar[l]_{\!\!\!\!\!\!\!\!\!\!\!\!\!\!\! \partial_{q }}
}
\]
\end{defn}

Let $\Lambda \subset \R$ be a discrete additive subgroup (possibly $\Lambda=0$).
Let $\tilde{\lambda} \in \AAA^i(\GG/K)^\GG$ be a $\GG$-invariant smooth $i$-form on $\GG/K$.
For a filling $\sigma$ over $\Ss$ we define
\begin{equation}\label{888s}
 \mathcal{F}_{\sigma}(\tilde{\lambda}) : S_i \longrightarrow \R/\Lambda ;  \qquad
 (g_0,\dots, g_i)\mapsto \int_{\sigma_i (g_0,\dots, g_i )} \tilde{\lambda}
 \quad \text{mod }\Lambda.
\end{equation}
\begin{defn}
A filling $\sigma$ over $\Ss$ is said to be
\emph{smooth modulo $\Lambda$} if for every $0\le i\le q$ and every
$\GG$-invariant $i$-form $\tilde{\lambda} \in \AAA^i(\GG/K)^{\GG}$ the map
$\mathcal{F}_{\sigma}(\tilde{\lambda})$ defined in~\eqref{888s} is smooth.
\end{defn}
\begin{rem}\label{aaa}
A smooth configured cochain is defined only on its restricted 
simplicial domain and need not be the restriction of a group cocycle.
For the \(U\)-small domains used below,
Proposition~\ref{baas} constructs, under homological torsion and
period hypotheses, a locally smooth group cocycle with the required
local restriction.
\end{rem}
When $S_i=\GG^{i+1}$ for all $i$, the notion of filling in Definition~\ref{33} reduces to
the construction originally introduced by Dupont \cite{Dup2}.
For proper subsets $S_n\subsetneq \GG^{n+1}$, the notion of a filling is not always stated explicitly in the literature, but concrete examples appear in
\cite{DK,DZ} and many related works (see also Examples~\ref{aa} and~\ref{i12} below).

If $i<q$ and $ d \tilde{\lambda}=0$, then $ \mathcal{F}_{\sigma}(\tilde{\lambda})$ is an $i$-cocycle; however, when $i=q$,
to ensure that $ \mathcal{F}_{\sigma}(\tilde{\lambda})$ is a $q$-cocycle, we need geometric conditions on $\Ss$ and $\Lambda$.
We give two examples. First, we recall Dupont's existence result in the homogeneous case.
\begin{prop}[{\cite[Theorem 10.2]{Dup2}, \cite[Theorems 5.2--5.3]{DK} }]\label{312}
Let $G$ be a Lie group and $K\subset G$ a closed subgroup such that $G/K$ is
$(q-1)$-connected and $H_q(G/K )\cong \Z$.
For the sequence $S_n=G^{n+1}$, there exists a filling
$\sigma_{G/K}:=\{\sigma_n\}$ of length $q$.
Moreover, any two such $\sigma_{G/K}$ are chain homotopic through degree $q-1$.

If $\lambda\in \AAA^q(G/K)^G$ represents a generator of $H^q_{\rm dR}(G/K)$ dual to a generator of $H_q(G/K ) \cong \Z$, then 
$\mathcal{F}_{\sigma_{G/K}}(\lambda)$ modulo $\Z$ is a
$\R/ \Z$-valued group $q$-cocycle.
Moreover, the $\R/\Z$-valued cohomology class $[\mathcal{F}_{\sigma_{G/K}}(\lambda)]$ 
does not depend on the chosen filling $\sigma_{G/K}$.
\end{prop}
The existence and smoothness of $\sigma$ in Proposition~\ref{312} depend on the geometry of $G/K$.
We next give another example.
The next example serves as the finite-dimensional model for the local
construction in the following subsection.

\begin{exa}[Dupont's geodesic straightening]\label{i12}
Let \(G\) be a semisimple Lie group and let \(K\subset G\) be a maximal
compact subgroup.  By the Cartan--Iwasawa decomposition, \(G/K\) is
diffeomorphic to a Euclidean space and carries a \(G\)-invariant Riemannian
metric of non-positive sectional curvature.  Thus \(G/K\) is a Hadamard
manifold.

For \(x,y\in G/K\), denote by
\(\operatorname{Join}(x,y;s)\) the point at time \(s\) on the unique
geodesic segment from \(x\) to \(y\).
Dupont's straight simplices are defined recursively by this join operation.
For \(n=0\), set
$
        \sigma_0(g_0)\equiv [g_0]\in G/K.
$
For \(n\ge1\), define
\(\sigma_n(g_0,\ldots,g_n):\Delta^n\to G/K\) by
\begin{equation}\label{Dupontsigma}
 \sigma_n(g_0,\ldots,g_n)\bigl((1-s)u,s\bigr)
 :=
 \mathrm{Join}\bigl(
   \sigma_{n-1}(g_0,\ldots,g_{n-1})(u),\,[g_n];\,s
 \bigr),
\end{equation}
where \(u\in\Delta^{n-1}\) and \(s\in[0,1]\).  The fact that these formulas
define smooth singular simplices on the closed simplices, with the usual
face identities, is part of Dupont's standard geodesic straightening
construction.

Thus \(\sigma=\{\sigma_n\}\) gives a smooth filling in the sense of
Definition~\ref{33}.  For invariant differential forms, integration over
these straight simplices gives Dupont's standard smooth homogeneous
cochains; in the classical setting this construction realizes the van Est
comparison; see \cite[Proposition~1.5]{Dup1} and also
\cite{Dup2,DK}.

The recursive join formula, rather than the global Riemannian geometry of
\(G/K\), is the feature that will be retained below.  For the
\(U\)-small domains in a Fr\'echet--Lie group, we replace the global
geodesic join by a smoothed affine join in a fixed convex chart, which is applicable to $\Diff_\lambda (M)$; see Proposition~\ref{baas}.
\end{exa}

\begin{rem}[The role of the configured framework]
\label{rem:configured-framework}
The preceding constructions separate two requirements that coincide in
Dupont's classical homogeneous setting but need not coincide on a
restricted domain.

Acyclicity is a homological condition ensuring comparison with
group cohomology.  A filling is geometric data used to integrate
differential forms smoothly and compatibly with the face maps.
The pairwise-distinct and generic-position examples illustrate the first
requirement, while Dupont's geodesic straightening illustrates the second.

The \(U\)-small domains introduced in the next section are local
configured domains equipped with canonical smooth fillings.  
We do not
assume that they are 
\(\mathbb Z[1/\nu]\)-acyclic. 
Instead,
Proposition~\ref{baas} uses torsion in the homology of the ambient group to
extend the resulting local configured cochains to locally smooth
homogeneous group cocycles.
\end{rem}

\section{Local affine fillings and torsion-controlled extensions for
\(\Diff_\lambda(M)\)}
\label{sec33233}

The preceding section presented Dupont's geodesic straightening as the
classical model for integrating invariant differential forms over
face-compatible simplices.  We now construct its local analogue for a
Fr\'echet--Lie group.

For a sufficiently small identity neighborhood \(U\), the domains
\(S_{n,\GG,U}\) consist of tuples whose relative differences lie in \(U\).
They form a configured domain containing a neighborhood of the diagonal.
Using a convex chart, we construct a smooth filling on this domain by
replacing Dupont's geodesic join with a smoothed affine join; see Section \ref{sec223344}.
We then extend the resulting local cochains to the full homogeneous complex
by a torsion argument (Section~\ref{secDiff0Configured}).

Throughout this paper, 
let $\GG$ be a Fr\'echet--Lie group with Lie algebra $\g:=T_e\GG$.

\subsection{Affine simplices in Fr\'echet--Lie groups}\label{sec223344}

The purpose of this subsection is to construct local face-compatible fillings
that work directly in the $C^\infty$ (Fr\'echet) setting.  The construction uses
only a fixed convex chart near the identity.

Fix a smooth chart
$\phi:U_0\stackrel{\cong}{\rightarrow}V_0\subset\g$ around the identity
$e\in\GG$ such that $\phi(e)=0$, $d\phi_e=\operatorname{id}_{\g}$, and $V_0$
is open and convex.

\begin{defn}\label{def:U-small}
For an identity neighborhood $U\subset\GG$ and $n\ge0$, define the following two sets:
\begin{align}
 S_{n,\GG,U}
 &:=
 \bigl\{(g_0,\dots,g_n)\in\GG^{n+1}
 \mid
 g_i^{-1}g_j\in U
 \ \text{for all }0\le i,j\le n
 \bigr\}, \label{ball} \\
 S^e_{n,\GG,U}
 &:=
 \bigl\{(h_1,\dots,h_n)\in\GG^n
 \mid
 (e,h_1,\dots,h_n)\in S_{n,\GG,U}
 \bigr\}. \notag
\end{align}
We call an element of \eqref{ball} a $U$-small $(n+1)$-tuple.
\end{defn}

\begin{prop}[Local face-compatible smooth filling]
\label{prop:local-affine-filling}
Fix an integer $m\ge0$.  
Then after choosing the open identity neighborhoods $U\subset\GG$ and $W \subset U_0$ sufficiently small, we have
\begin{equation}\label{Uconditions}
 e\in U,\qquad U\subset U_0. 
\end{equation}
and, for each $0\le n\le m$, there exist smooth maps
\[
 \sigma_n:S_{n,\GG,U}\times\Delta^n\longrightarrow\GG
\]
with the following properties.

\begin{enumerate}
 \item[\textup{(i)}] Vertex condition and image control: for every $0\le n\le m$, every
 $0\le i\le n$, and every $(g_0,\dots,g_n)\in S_{n,\GG,U}$, we have
 $\sigma_n(g_0,\dots,g_n)(v_i)=g_i$, and
 \[
  \sigma_n(g_0,\dots,g_n)(\Delta^n)\subset g_0W.
 \]
\item[\textup{(ii)}]
Left equivariance and face compatibility:
for every \(0\leq i \leq n\leq m\) and \(h\in\GG\),
\begin{align}
 \sigma_n(hg_0,\ldots,hg_n)
 &=
 h\,\sigma_n(g_0,\ldots,g_n), \notag \\
 \left.
 \sigma_n(g_0,\ldots,g_n)
 \right|_{d_i\Delta^{n-1}}
 &=
 \sigma_{n-1}(g_0,\ldots,\widehat g_i,\ldots,g_n). \label{ddde}
\end{align}
\end{enumerate}
Consequently, $\Ss_{\GG,U}:=\{S_{n,\GG,U}\}_{0\le n\le m}$ satisfies the simplicial condition,
and $\sigma=\{\sigma_n\}_{0\le n\le m}$ gives a filling over $\Ss_{\GG,U}$.
Moreover, for each $0\le n\le m$ and each $\GG$-invariant $n$-form
$\tilde{\lambda} \in\Omega^n(\GG)^\GG$, the following map
is a smooth function on $S_{n,\GG,U}$:
\begin{equation}\label{gyou} 
 (g_0,\dots,g_n)
 \mapsto
 \int_{\sigma_n(g_0,\dots,g_n)}\tilde{\lambda} 
\end{equation} 
\end{prop}

\begin{proof}
Choose an open identity neighborhood $W\subset U_0$ such that
$W^{-1}W\subset U_0$.  Fix a smooth function
$\rho:[0,1]\to[0,1]$ which is $0$ near $0$ and $1$ near $1$.  Whenever
$x^{-1}y\in U_0$, set
\[
{\rm Join}_\phi(x,y;s)
 :=
 x\,\phi^{-1}\!\bigl(\rho(s)\phi(x^{-1}y)\bigr).
\]
This is well-defined because $V_0$ is convex and contains $0$; it is smooth
and left equivariant.

Set $\sigma_0(g_0)=g_0$.  We construct the higher maps inductively,
shrinking the identity neighborhood at each stage.  Suppose that, for an
identity neighborhood $V$, the maps $\sigma_k$, $k<n$, have been constructed
on $S_{k,\GG,V}$, satisfy the required vertex, equivariance, and face
conditions, and satisfy $\sigma_k(e,\ldots,e)\equiv e$.

On $S^e_{n,\GG,V}\times\Delta^{n-1}$ consider the smooth map
\[
 \overline F_n(h_1,\ldots,h_n;u)
 :=
 \sigma_{n-1}(e,h_1,\ldots,h_{n-1})(u)^{-1}h_n.
\]
It is identically $e$ on
$\{(e,\ldots,e)\}\times\Delta^{n-1}$.  Since $\Delta^{n-1}$ is compact,
there is an identity neighborhood $V'\subset V$ such that
$(V')^n\subset S^e_{n,\GG,V}$ and
$\overline F_n((V')^n\times\Delta^{n-1})\subset U_0$.

Let $(g_0,\ldots,g_n)\in S_{n,\GG,V'}$ and put
$h_i=g_0^{-1}g_i$.  Left equivariance of $\sigma_{n-1}$ gives
\[
 \sigma_{n-1}(g_0,\ldots,g_{n-1})(u)^{-1}g_n
 =\overline F_n(h_1,\ldots,h_n;u)\in U_0.
\]
In the cone coordinates $((1-s)u,s)\in\Delta^n$, define, for $0\leq s<1$,
\[
 \sigma_n(g_0,\ldots,g_n)((1-s)u,s)
 :=
{\rm Join}_\phi \!\left(
   \sigma_{n-1}(g_0,\ldots,g_{n-1})(u),g_n;s
 \right),
\]
and set $\sigma_n(g_0,\ldots,g_n)(v_n)=g_n$.

Because $\rho=1$ near $1$, this formula is independent of $u$ near the cone
vertex and hence extends jointly smoothly across $v_n$.  Because $\rho=0$
near $0$, its restriction to the base face is
$\sigma_{n-1}(g_0,\ldots,g_{n-1})$.  On every other face, the induction
hypothesis and the use of the same function $\rho$ give \eqref{ddde}.  The
vertex conditions and left equivariance follow at once, and
${\rm Join}_\phi(e,e;s)=e$ gives $\sigma_n(e,\ldots,e)\equiv e$.

After the finite induction, restrict all maps to the final identity
neighborhood $V_*$.  For each $1\leq n\leq m$, joint smoothness and compactness
of $\Delta^n$ give a neighborhood $\mathcal O_n$ of $(e,\ldots,e)$ in
$S^e_{n,\GG,V_*}$ such that
$\sigma_n(e,h_1,\ldots,h_n)(\Delta^n)\subset W$ whenever
$(h_1,\ldots,h_n)\in\mathcal O_n$.  Choose one identity neighborhood
$U\subset V_*$ such that $U^n\subset\mathcal O_n$ for all
$1\leq n\leq m$, and shrink it further so that
$U\subset U_0$ and $U^{-1}U\subset U_0$.

If $(g_0,\ldots,g_n)\in S_{n,\GG,U}$, then
$(g_0^{-1}g_1,\ldots,g_0^{-1}g_n)\in U^n$, and left equivariance yields
\[
 \sigma_n(g_0,\ldots,g_n)(\Delta^n)
 =g_0\,\sigma_n(e,g_0^{-1}g_1,\ldots,g_0^{-1}g_n)(\Delta^n)
 \subset g_0W.
\]
The case $n=0$ is immediate.  Since the sets $S_{n,\GG,U}$ are invariant
under the diagonal action and the face maps, the linear extensions of the
$\sigma_n$ form the required filling.

Finally, $\sigma_n$ is jointly smooth in the tuple and simplex variables.
Smooth parameter-dependent integration over the compact simplex therefore
shows that the function in \eqref{gyou} is smooth; see
\cite[\S3.15]{KrieglMichor}.
\end{proof}


\subsection{Torsion-controlled extensions for \(\Diff_\lambda(M)\)}
\label{secDiff0Configured}
Fix an integer $q\geq2$ and assume that $\Diff_\lambda(M)$ is a
Fr\'echet--Lie subgroup of $\Diff(M)$ whose identity component admits a
convex identity chart as in Proposition~\ref{prop:local-affine-filling}.  We
use the natural comparison between smooth and ordinary singular homotopy and
homology for the Fr\'echet manifolds considered here.  Thus every homotopy
class represented by a map from a finite-dimensional sphere has a smooth
representative, and a homological relation among smooth singular cycles can
be realized by a smooth singular chain (see \cite[Corollary~11.3]{Kihara} and \cite[Theorem~0.11]{Wockel} for details).

We begin with the following hypothesis.
\begin{defn}\label{assume}
 We say that
$(M,\lambda)$ satisfies assumption $(\dagger)_q$ if
$\pi_0(\Diff_\lambda(M))$ is finite, $\nu_M<\infty$, and
$\pi_j(\Diff_\lambda(M))=0$ for $1<j<q$.

In what follows, we denote the subgroup $\Diff_{\lambda,0}(M)$ by $\GG_M$ and $| \pi_1(\GG_M )|\in \N $ by $\nu_M$.
\end{defn}
The finiteness of $\pi_0(\Diff_\lambda(M))$ is included for the later transfer
arguments and is not used in Proposition~\ref{baas}.  The required homotopy
information will be recalled in each application.
As mentioned in Sections~\ref{sec3a884} and~\ref{sec3a8842}, 
the homotopy types of the groups in Theorems~\ref{main1}--\ref{main3} 
are given in Table~\ref{tab:manifold-groups}; thus, these situations satisfy $(\dagger)_3 $.
\begin{table}[htbp]
\centering
\begin{tabular}{lllll}
\hline
$(M,\lambda) $ & Homotopy type of $\GG_0$ & $|\pi_1|$ & $\pi_2$ & $\pi_3$ \\ \hline
$(S^3,v_{S^3})$ & $\SO(4)$ & 2 & 0 & $\mathbb{Z}^2$ \\
$(\text{noncyclic } S^3/\Gamma,v) $ & $\SO(3)$ & $2$ & 0 & $\mathbb{Z}$ \\
$(\mathbb{C}P^n ,\omega_n)$ & $\mathrm{PSU}(n+1)$ & $n+1$ & 0 & $\mathbb{Z}$ \\
$(\mathbb{C}P^1 \times \mathbb{C}P^1, \omega_1 \oplus \omega_1 )$ & $\SO(3) \times \SO(3)$ & 4 & 0 & $\mathbb{Z}^2$ \\ \hline
\end{tabular}
\caption{Types of $\GG_0$ and homotopy groups for each manifold $M$ with $\GG= \Diff_{\lambda}(M)$. Here $n=1,2$ in the third line.}
\label{tab:manifold-groups}
\end{table}

The purpose of this section is to show Proposition~\ref{baas}, which is a key result.

\begin{lem}[Homological torsion]\label{key3}
If $(M,\lambda)$ satisfies $(\dagger)_q$, then
\[
 \nu_M H_j(\GG_M;\mathbb Z)=0\qquad(1\leq j<q),
\]
and the cokernel of
$\operatorname{hur}_q:\pi_q(\GG_M)\to H_q(\GG_M;\mathbb Z)$ is annihilated
by $\nu_M$.
\end{lem}

\begin{proof}
Let $p:\widetilde\GG_M\to\GG_M$ be the universal covering, which is
$\nu_M$-sheeted.  As shown in \cite{Kihara,Wockel}, 
the Hurewicz theorem holds in the class of diffeomorphism groups of closed manifolds;
The total space $\widetilde\GG_M $ is $(q-1)$-connected, so the Hurewicz theorem gives
$H_j(\widetilde\GG_M;\mathbb Z)=0$ for $0<j<q$ and an isomorphism
$\pi_q(\widetilde\GG_M)\cong H_q(\widetilde\GG_M;\mathbb Z)$.

Let $\operatorname{tr}$ be the transfer.  Since
$p_*\operatorname{tr}=\nu_M\operatorname{id}$ by the transfer identity, the vanishing above gives
$\nu_M H_j(\GG_M;\mathbb Z)=0$ for $1\leq j<q$.  For
$x\in H_q(\GG_M;\mathbb Z)$, choose
$\widetilde\alpha\in\pi_q(\widetilde\GG_M)$ with
$\operatorname{hur}_q(\widetilde\alpha)=\operatorname{tr}(x)$.
Naturality of the Hurewicz map then gives
$\nu_Mx=\operatorname{hur}_q(p_\#\widetilde\alpha)$.
\end{proof}

\begin{prop}[Torsion-controlled extension]\label{baas}
Assume that $(M,\lambda)$ satisfies $(\dagger)_q$.  Choose $U$ and the local
filling
$\sigma_j:\mathbb Z\langle S_{j,\GG_M,U}\rangle\to
C_j^{\rm sm}(\GG_M)$, $0\leq j\leq q$, supplied by
Proposition~\ref{prop:local-affine-filling}.  Let
$\widetilde\lambda\in\AAA^q(\GG_M)^{\GG_M}$ be closed, and let
$\Lambda\subset\mathbb R$ be a discrete additive subgroup.  Suppose that
$\int_{S^q}c^*\widetilde\lambda\in\Lambda$ for smooth based
representatives $c:(S^q,s_0)\to(\GG_M,e)$ of a set of generators of
$\pi_q(\GG_M)$.

Then there exists a $\GG_M$-invariant locally smooth homogeneous group
$q$-cocycle $\Phi:\GG_M^{q+1}\to\mathbb R/\Lambda$ such that
\[
 \left.\Phi\right|_{S_{q,\GG_M,U}}
 =\nu_M^q\mathcal F_\sigma(\widetilde\lambda).
\]
\end{prop}

\begin{proof}
Write $G=\GG_M$ and $\nu=\nu_M$.  Since $\widetilde\lambda$ is closed,
integration over based $q$-spheres factors through the Hurewicz map and
defines a homomorphism on $\pi_q(G)$.  Thus the period condition holds for
every element of $\pi_q(G)$.

We construct $G$-equivariant
homomorphisms $T_j:C_j^\Delta(G)\to C_j^{\rm sm}(G)$, $0\leq j\leq q$,
with
\begin{equation}\label{eq:torsion-twisted-filling}
\begin{aligned}
 T_0(g)&=[g],
 &T_j\big|_{\mathbb Z\langle S_{j,G,U}\rangle}
   &=\nu^{j-1}\sigma_j &&(1\leq j\leq q),\\
 \partial T_1&=T_0\partial,
 &\partial T_j&=\nu T_{j-1}\partial &&(2\leq j\leq q).
\end{aligned}
\end{equation}
The diagonal action of $G$ on $G^{j+1}$ is free, and $S_{j,G,U}$ is a union
of orbits.  Choose one representative from each orbit in
$G^{j+1}\setminus S_{j,G,U}$; all choices below are extended
$G$-equivariantly and $\mathbb Z$-linearly.

Set $T_0(g)=[g]$ and prescribe the local values in
\eqref{eq:torsion-twisted-filling}.  Face compatibility of $\sigma$ gives the
required boundary identities on the local submodules.  For a chosen tuple
$b=(g_0,g_1)\notin S_{1,G,U}$, the cycle $[g_1]-[g_0]$ is zero in
$H_0(G;\mathbb Z)$ because $G$ is connected.  Smooth singular comparison
therefore gives a smooth $1$-chain $T_1(b)$ with
$\partial T_1(b)=[g_1]-[g_0]$.

Inductively, let $2\leq j\leq q$ and suppose that $T_0,\ldots,T_{j-1}$ have
been constructed.  For a chosen tuple
$b\in G^{j+1}\setminus S_{j,G,U}$, put $z_b=T_{j-1}(\partial b)$.  The
previous boundary identity and $\partial^2=0$ show that $z_b$ is a smooth
$(j-1)$-cycle.  Since $j-1<q$, Lemma~\ref{key3} gives
$\nu[z_b]=0$ in $H_{j-1}(G;\mathbb Z)$.  Smooth singular comparison now
provides a smooth $j$-chain $T_j(b)$ with
$\partial T_j(b)=\nu z_b$.  This completes the induction and proves
\eqref{eq:torsion-twisted-filling}.  Notice that the complement of
$S_{j,G,U}$ need not be closed under faces: the construction uses only that
$T_{j-1}$ is already defined on all of $C_{j-1}^\Delta(G)$.  In particular,
no acyclicity of the local subcomplex is used.

For $(g_0,\ldots,g_q)\in G^{q+1}$, define
\[
 \Phi(g_0,\ldots,g_q)
 :=
  \nu\int_{T_q(g_0,\ldots,g_q)}\widetilde\lambda \in \mathbb R/\Lambda.
\]
Equivariance of $T_q$ and left invariance of $\widetilde\lambda$ make $\Phi$
invariant.  Let $y\in G^{q+2}$ and put $z=T_q(\partial y)$.  Since $q\geq2$,
\eqref{eq:torsion-twisted-filling} gives $\partial z=0$.  By
Lemma~\ref{key3}, there is $\alpha\in\pi_q(G)$ such that
$\operatorname{hur}_q(\alpha)=\nu[z]$.  Choose a smooth based representative
$c:S^q\to G$ of $\alpha$ and a smooth singular fundamental cycle
$\zeta_q\in C_q^{\rm sm}(S^q)$.  Smooth singular comparison yields a smooth
$(q+1)$-chain $a$ satisfying $\partial a=\nu z-c_*\zeta_q$.  Stokes' theorem
and the period assumption give
\[
 (\partial^q\Phi)(y)
 = \nu\int_z\widetilde\lambda 
 =\int_{S^q}c^*\widetilde\lambda 
 =0  \in \mathbb R/\Lambda.
\]
Thus $\Phi$ is a homogeneous group cocycle.

For a tuple in $S_{q,G,U}$, the local prescription in
\eqref{eq:torsion-twisted-filling} gives
$T_q=\nu^{q-1}\sigma_q$, and hence
$\Phi=\nu^q\mathcal F_\sigma(\widetilde\lambda)$ there.  The latter is smooth
on the open neighborhood $S_{q,G,U}$ of $(e,\ldots,e)$ by
Proposition~\ref{prop:local-affine-filling}.  Therefore $\Phi$ is locally
smooth.
\end{proof}


\subsection{Relation to Gel'fand--Fuks cohomology}\label{Fin.sec}
In Proposition~\ref{baas}, we presented a procedure for constructing locally smooth $q$-cocycles from a $q$-form $\tilde{\lambda}$; 
here, we discuss its relation to Gel'fand--Fuks 
cohomology, and show that it recovers the form $\tilde{\lambda} $ (Proposition~\ref{43333}).
Furthermore, for $M=S^3$, the non-extendability of the locally smooth $3$-cocycle to $\Diff^+(M)$ will be shown in \S \ref{Diff+(S3)}.

Let $ \GG$ be a Fr\'echet--Lie group with Lie algebra $\mathfrak g$,
and suppose that $\GG$ admits a smooth local exponential map $\g \ra \GG$.
For $n\ge0$, we define the smooth Lie algebra complex with trivial coefficients
$\R$ by
\[
 C^n_{\rm GF}(\mathfrak g)
 :=
 \{ \omega:  \mathfrak g^n  \longrightarrow \R \mid
    \omega \text{ is alternating, multilinear, and of class } C^{\infty} \}.
\]
The differential $d_{\mathfrak g}: C^n_{\rm GF}(\mathfrak g) \to C^{n+1}_{\rm GF}(\mathfrak g)$
is given by the usual Chevalley--Eilenberg formula
\[
  (d_{\mathfrak g}\omega)(X_1,\dots,X_{n+1})
= \sum_{i,j: 1 \leq i < j \leq n+1} (-1)^{i+j}\,
    \omega([X_i,X_j], X_1,\dots,\widehat X_i,\dots,\widehat X_j,\dots,X_{n+1} ),
\]
for $X_1,\dots,X_{n+1}\in\mathfrak g$.
It is well known that $d_{\mathfrak g}\circ d_{\mathfrak g}=0$.
The resulting cohomology $H^*_{\rm GF}(\mathfrak g)$ is called the continuous Chevalley--Eilenberg cohomology. If $ \mathfrak g$ is a subalgebra of a vector-field algebra, $H^*_{\rm GF}(\mathfrak g)$  is called the Gel'fand--Fuks cohomology.
It is canonically isomorphic to the cohomology of the complex $(\AAA^*(\GG)^{\GG},d)$ of left-invariant differential forms on $ \GG$;
see, for example, \cite{Fuks}.

Next, we recall a chain map from locally smooth group cochains to
$C^*_{\rm GF}(\mathfrak g)$.
Let 
$A$ be a topological abelian group isomorphic either to $\R$ or to $\R/\Z$, with trivial
$ \GG$-action.
For $f\in C^n_{\rm loc,sm}(\GG;A)^{\GG}$ with $n>0$, we choose a neighborhood $U_f\subset  \GG^{n+1}$
of $e^{\times (n+1)}$ on which $f$ is smooth.
Since $\exp(0)=e$ and $\exp$ is smooth, after shrinking $U_f$ we may choose a
$0$-neighborhood $V_f\subset \mathfrak g$ such that $\bigl(\exp(V_f)\bigr)^{n+1}\subset U_f$.
When $A=\R/\Z$, 
after shrinking $U_f$, its image lies in an evenly covered arc of $\R/\Z$; 
hence $f |_{U_f}$  admits a smooth real-valued lift $\tilde{f}$.
When $A=\R$, we set $\tilde{f}:=f$.
For $X_1,\dots,X_n\in\mathfrak g$, we set
\[
 \Phi_f(X_1,\dots,X_n)
 :=
 \tilde{f} \bigl(
     e,\,
     \exp(X_1),\,
     \exp(X_1)\exp(X_2),\,
     \dots,\,
     \exp(X_1)\cdots\exp(X_n)
   \bigr).
\]
Then $\Phi_f$, regarded as a function on $\mathfrak g^n$, is smooth in a neighborhood of the origin
$(0,\dots,0)$, and we may define the iterated directional derivative
\begin{equation*}
\bigl(T^n f\bigr)(X_1,\dots,X_n)
:=  \frac{\partial^n}{\partial t_1\cdots \partial t_n}\,
\Phi_f(t_1X_1,\dots,t_nX_n)\bigr|_{t_1=\cdots=t_n=0}.
\end{equation*}
Next, we define the alternating $n$-linear map
\[
D(f)(X_1,\dots,X_n)
:=
\sum_{\sigma \in \mathfrak{S}_n} \mathrm{sgn}(\sigma)\,
\bigl(T^n f\bigr)(X_{\sigma(1)},\dots,X_{\sigma(n)}),
\qquad X_1,\dots,X_n\in\mathfrak g.
\]
Then $D(f)\in C^n_{\mathrm{GF}}(\mathfrak g)$, and 
the assignment $D$ is the standard differentiation map from locally smooth
homogeneous group cochains to continuous Chevalley--
Eilenberg cochains and is a cochain map.

The next proposition shows how the cochains $\mathcal{F}_{\sigma}(\tilde{\lambda} )$ in \eqref{888s} 
give rise, under the map $D$, to invariant forms.
\begin{prop}\label{43333}

Let $\GG$ be a Fr\'echet--Lie group and fix $U$ and the smooth filling
$\sigma$ on $\Ss_{\GG,U}$ from Section~\ref{sec223344}.
Let $\tilde{\lambda}\in\AAA^i(\GG)^\GG$ be a closed
$\GG$-invariant $i$-form with $i\geq1$.
Suppose $\Lambda=0$ or that $\Lambda$ is a lattice isomorphic to $\Z$.
Recall from \eqref{888s} the locally smooth $i$-cochain
$\mathcal F_\sigma(\tilde{\lambda})$ defined on
$S_{i,\GG,U}$ with values in $\R/\Lambda$.

Suppose that
$c\,\mathcal F_\sigma(\tilde{\lambda})$
extends to a locally smooth homogeneous group $i$-cocycle
\[
 \Phi:\GG^{i+1}\longrightarrow\R/\Lambda
\]
for some $c\in\N$.
Then
\begin{equation}\label{hoshii}
 D(\Phi)=c\,\tilde{\lambda}_e
 \quad\text{in}\quad
 C^i_{\rm GF}(\mathfrak g) \cong\AAA^i(\GG)^\GG.
\end{equation}
\end{prop}
For completeness, we include the proof in Appendix~\ref{furoku}; the argument is somewhat lengthy.
\begin{rem}
Although we use it only once, Proposition~\ref{43333}
recovers the usual comparison between continuous group cohomology and Lie algebra cohomology
in the finite-dimensional exponential case.
More precisely, if $\mathfrak g$ is finite-dimensional and
$\exp:\mathfrak g\to  \GG$ is a global diffeomorphism, then $ \GG$ is an exponential Lie group, and the van Est theorem implies
$
H^n_{\rm loc,sm}(\GG;\R)\cong H^n_{\rm GF}(\mathfrak g)
$
(equivalently, by \cite[Proposition~1.5]{WW}, $H^n_{\rm sm}(\GG;\R)\cong H^n_{\rm loc,sm}(\GG;\R)$, so $H^n_{\rm sm}(\GG;\R)\cong H^n_{\rm GF}(\mathfrak g)$).
Thus Proposition~\ref{43333} may be viewed as a cocycle-level realization of this comparison in the exponential setting.
\end{rem}

\section{Applications to the locally smooth cohomology of $\Diff_v(M)$}\label{sec2233}
The goal of this section is to prove Theorems~\ref{main1} and~\ref{main2}.
The key ingredient is the Chern--Simons invariant, which we briefly review in
Section~\ref{sssdd}.
In Section~\ref{sec3a884}, we prove Theorem~\ref{main2}, and in Section~\ref{hyp11} we also discuss the hyperbolic case.


\subsection{Review of $\SU(2)$ Cheeger--Chern--Simons cocycles}
\label{sssdd}

We briefly recall the Chern--Simons transgression and the corresponding
Cheeger--Chern--Simons (CCS) 3-cocycles; see \cite{CS,Dup2,DK,DZ}.
Chern--Simons theory and its invariants form a broad subject, and a general account would require a long discussion.
Therefore, in this subsection we record only the facts about the Cheeger--Chern--Simons cocycle, a 3-cocycle $c_P^{\mathrm{CS}}: G^4 \ra \C/\Z$ that will be used later.
Here, \(G=\SL_2(\C)\) or
\(G=\SU(2)\).

First, consider the invariant polynomial corresponding to the second Chern class:
\[P: \g \times \g \longrightarrow \C; \qquad (A,B) \mapsto \frac{1}{8\pi^2}\Tr(AB).\]
Let $\omega_{\rm MC} \in \Omega^1(G;\g)$ be the Maurer--Cartan form. Define
\begin{equation}\label{mc3}
\lambda_{\rm MC}:=\frac{-1}{24\pi^2}\Tr(\omega_{\rm MC}\wedge\omega_{\rm MC}\wedge\omega_{\rm MC}) \in \Omega^3(G)^G
 \end{equation} 
as a $G$-invariant $3$-form on $G$.
The form $\lambda_{\rm MC}$ will recur below. 
If $G=\SU(2)$, then the image of $c_P^{\mathrm{CS}}$ is $\R/\Z$, and under the diffeomorphism $\SU(2)\cong S^3$, $\lambda_{\rm MC}$ is equal to the 
normalized bi-invariant volume form
$v_{S^3}$ with $\int_G v_{S^3}=1$. 
The following definition and fact will be used repeatedly.
\begin{defn}\label{tsukau4}
Let 
\(G=\SU(2)\).
Since $G$ is $2$-connected and $\pi_3(G) \cong \Z$, Proposition~\ref{312} with $K=\{e\}$ and $q=3$ gives a filling $\sigma_G$ on the homogeneous
sequence $S_n=G^{n+1}$.
From Theorem \ref{tsukau} (or Proposition \ref{baas}), take the resulting locally smooth 3-cocycle $\Phi: G^4\to \R/\Z$, such that the 3-class of the derivative is $[(\lambda_{\rm MC})_e]$. 

Hereafter, we denote by $c_P^{\mathrm{CS}}$ the 3-cocycle $\Phi$.
\end{defn}
\begin{thm}\label{tsukau}
Then, the 3-cocycle 
represents a CCS 3-class $[c_P^{\mathrm{CS}}]$ associated with the second Chern class.

Furthermore, for $m\ge2$, consider the injective (rotation) homomorphism
\begin{equation}\label{jmjm}
R_m:\Z/m\longrightarrow \SU(2),\qquad
a\mapsto
\begin{pmatrix}
\exp(2\pi a\sqrt{-1}/m) & 0\\
0 & \exp(-2\pi a\sqrt{-1}/m)
\end{pmatrix},
\end{equation}
Then there exists a generator $\tau_m$ of $H_3^{\rm gr}(\Z/m;\Z)\cong \Z/m$ for every $m \geq 2$ such that
the compatibility relation $(R_m)_*(\tau_m)= k (R_{km})_*(\tau_{km})$ and  
\begin{equation}\label{eq:CCS-cyclic-evaluation}
\bigl\langle [c_P^{\mathrm{CS}}],(R_m)_*\tau_m\bigr\rangle=[1/m]\in \Q/\Z. 
\end{equation}
\end{thm}

The construction of $[c_P^{\mathrm{CS}}]$ goes back to \cite{Dup2,DK,DZ} and is related to the Rogers dilogarithm and the (extended) Bloch groups, where $ c_P^{\mathrm{CS}}$ can be extended to an $\SL_2(\C)$-Chern--Simons 3-class. The proof appears, e.g., in \cite[Section 1.3]{DZ}; we only use the facts stated above.
The reader should keep in mind the rotation homomorphism $R_m $ defined above, as it will reappear later.

\subsection{Lemmas on transfer and CCS cocycle comparison}\label{sec3a4}

We will use Lemma~\ref{321a} on transfer in group cohomology, and 
a comparison lemma about the Cheeger--Chern--Simons 3-cocycle several times.
For completeness, we recall them here.
Let $\Gamma$ be a group, let $G\subset\Gamma$ be a normal subgroup of finite
index, and let $A$ be either $\R$ or $\R/\Z$, regarded as a trivial
$\Gamma$-module.

We review the transfer (corestriction) on homogeneous cochains from \cite[Chapter 4]{Bro}, and prove Lemma~\ref{321a}.
Choose a set $T\subset \Gamma$ of representatives for the \emph{right} cosets $G\backslash\Gamma$;
thus $\Gamma=\bigsqcup_{t\in T}Gt$ and we may assume $e\in T$.
For $t\in T$ and $\gamma\in \Gamma$ let $\overline{t\gamma}\in T$ be the unique representative of the right coset $G(t\gamma)$.
For $\phi \in C^n_{\Delta}(G;A)$, we define a cochain $\mathrm{cor}^{\Gamma}_G(\phi)\in C^n_{\Delta}(\Gamma;A)$ by
\begin{equation}\label{eq:transfer-def}
\mathrm{cor}^{\Gamma}_G(\phi)(\gamma_0,\ldots,\gamma_n)
:=\sum_{t\in T}\phi\bigl((t\gamma_0) \overline{t\gamma_0}^{-1}, (t\gamma_1 ) \overline{t\gamma_1}^{-1} , \ldots, (t\gamma_n ) \overline{t\gamma_n }^{-1} \bigr).
\end{equation}
The map $\mathrm{cor}^{\Gamma}_G:\ C^*_{\Delta}(G;A)^{G}\ra C^*_{\Delta}(\Gamma;A)^{\Gamma}$ is a cochain map.
If the cocycle $\phi$ is invariant under the conjugation action of $\Gamma/G$, we have $j^* ([ \mathrm{cor}^{\Gamma}_G(\phi)])=|\Gamma/G|\,[\phi]$ by the definition of transfer.
Dually, we can define the transfer $\mathrm{tr}^{\Gamma}_G: H_n^{\rm gr}(\Gamma) \ra  H_n^{\rm gr}(G)$ as a homomorphism on homology; see \cite{Bro}. 
We next prove Lemma~\ref{321a}.

\begin{lem}\label{321a}
Let 
$G\triangleleft\Gamma$ be a subgroup of finite index.
Suppose that $\phi\in Z^n(G;A)^G$ is a homogeneous 
group $n$-cocycle.
\begin{enumerate}[(i)]

\item 
Suppose that $\Gamma$ is 
a (Fr\'echet) Lie group, and $G$ is an open
normal subgroup. If $ \phi$ is locally smooth, then so is $ \mathrm{cor}^{\Gamma}_G(\phi)\in Z^n_{\rm loc,sm}(\Gamma;A )^\Gamma $.
\item Let $n$ be odd, and let $A=\R/\Z$. 
For each $m$, choose a compatible generator
$\tau_m\in H_n^{\rm gr}(\Z/m;\Z)\cong\Z/m$ in the standard direct-limit system. 
Fix 
a nonzero integer
$c_G\in\Z$ such that, for every $m\in\N$, there exists a group homomorphism
$\iota_m:\Z/m\to G$ compatible with the standard direct-limit system, namely,
$\iota_{km}([r])^k=\iota_m([r])$ for all $r$, and
\[\langle   \iota_m^*[\phi] ,  \tau_m \rangle=[c_G/m]\in \Q/\Z
\]
is satisfied. 
Then $H_n^{\rm gr}(G;\Z)$ contains a subgroup
$A_G\cong\Q/\Z$.

\item If, in addition, the induced action of $\Gamma/G$ on $A_G$ is trivial,
then $H_n^{\rm gr}(\Gamma;\Z)$ contains a subgroup isomorphic to $\Q/\Z$.

In particular, when $n=3$, if the action of $\Gamma/G$ on $ \mathrm{Im}(\iota_m)$ is generated by
the automorphism $ \iota_m(a) \mapsto \iota_m(a^{-1})  $, 
then the action of $\Gamma/G$ on $A_G$ is trivial, and $\Q/\Z \hookrightarrow H_3^{\rm gr}(\Gamma;\Z)$. 
\end{enumerate}
\end{lem}

\begin{proof}
\noindent\textbf{(i)}
Since \(\phi\) is locally smooth and \(T\) is finite, we may choose an
identity neighbourhood \(U\subset G\) so small that \(\phi\) is
smooth on \((tUt^{-1})^{n+1}\) for every \(t\in T\).  Since \(G\) is
open in \(\Gamma\), \(U\) is also an identity neighbourhood in
\(\Gamma\).  For \(\gamma_0,\ldots,\gamma_n\in U\), we have
\(\overline{t\gamma_j}=t\). Hence, by the definition of $ \operatorname{cor}^{\Gamma}_G$,
$ \operatorname{cor}^{\Gamma}_G$
is smooth on \(U^{n+1}\), being a finite sum of smooth maps.

\noindent\textbf{(ii)}
Consider the evaluation map ${\rm ev}:=\langle[\phi],-\rangle:\Tor H_n^{\rm gr}(G ;\Z)\to \Tor(\R/\Z)=\Q/\Z$.
The compatibility of the homomorphisms \(\iota_m\) gives a homomorphism
\(\alpha:\Q/\Z=\varinjlim_m H_n^{\rm gr}(\Z/m;\Z)\to
\Tor H_n^{\rm gr}(G;\Z)\).  For the class represented by \(s\tau_m\), the
composite \({\rm ev}\circ\alpha\) takes the value \([c_Gs/m]\).  Hence
\({\rm ev}\circ\alpha\) is multiplication by \(c_G\) on \(\Q/\Z\).  Its kernel
is the finite subgroup \((\Q/\Z)[c_G]\cong\Z/c_G\Z\), and its image is all of
\(\Q/\Z\).  Therefore \(\Ker\alpha\) is finite, and
\(\alpha(\Q/\Z)\cong(\Q/\Z)/\Ker\alpha\cong\Q/\Z\).  This gives the required
subgroup \(\Q/\Z\hookrightarrow H_n^{\rm gr}(G;\Z)\).

\noindent\textbf{(iii)} Let $j : G \hookrightarrow \Gamma$ be the inclusion, and let $ j_*^A $ be the restriction of $j_*: H_n^{\rm gr}(G;\Z) \ra H_n^{\rm gr}(\Gamma;\Z) $ to $A_G$. 
By assumption on the trivial action on $A_G$ and the definition of $\mathrm{res}^{\Gamma}_G$, we have $\mathrm{res}^{\Gamma}_G \circ j_*|_{A_G}
=[\Gamma:G]\operatorname{id}_{A_G}.$ 
In particular, the kernel of $j_*^A$ is finite, and $\ker(j_*|_{A_G})\subset |[\Gamma:G]A_G $.
Since $(\Q/ \Z)/\Ker(j_*^A ) $ is isomorphic to $\Q/\Z$, the image of $j_*^A$ gives the required subgroup. 
\end{proof}

\begin{rem}\label{rem:transfer-no-lift}
The classes detected by Lemma~\ref{321a}\textup{(ii)} are genuinely
\(\R/\Z\)-valued.  Indeed, if one of them were cohomologous to the reduction
modulo \(\Z\) of an \(\R\)-valued \(n\)-cocycle, then its restriction to
\(\Z/m\) would lie in the image of \(H^n(\Z/m;\R)\rightarrow H^n(\Z/m;\R/\Z)\).
This image is zero for \(n>0\), since \(H^n(\Z/m;\R)=0\).  This contradicts
the nonzero cyclic evaluations supplied by Lemma~\ref{321a}\textup{(ii)}
for suitable \(m\).  Thus the torsion-detection argument essentially uses
\(\R/\Z\)-coefficients.
\end{rem}

We shall also use the following comparison lemma repeatedly below.
\begin{lem}[Cheeger--Chern--Simons comparison]
\label{lem:comparison-CCS}
Assume that \((M,\lambda)\) satisfies \((\dagger)_3\), and put
\(\GG_M=\Diff_{\lambda,0}(M)\).
Let
\(\widetilde\lambda\in\AAA^3(\GG_M)^{\GG_M}\)
be closed and satisfy the period hypothesis of
Proposition~\ref{baas} with \(\Lambda=\Z\).
Let
\(\Phi\in Z^3_{\rm loc,sm}(\GG_M;\R/\Z)\)
be any cocycle supplied by that proposition.

Suppose that a smooth group homomorphism
\(L:\SU(2)\to\GG_M\) satisfies
\[
 L^*\widetilde\lambda
 =
 d\,\lambda_{\rm MC}
\]
for some \(d\in\Z\).
Then, in the locally smooth cohomology of degree 3, the following equality holds: 
\begin{equation}\label{hphp}
 L^*[\Phi]
 =
 d\nu_M^3[c_P^{\rm CS}]
 \quad\text{in}\quad
 H^3_{\rm loc,sm}(\SU(2);\R/\Z).
\end{equation}
The same identity holds after passage to
\(H^3_{\rm gr}(\SU(2);\R/\Z)\).
Moreover, for every \(m\geq2\),
\begin{equation}\label{eq:CCS-comparison-cyclic}
 \left\langle
 (L\circ R_m)^*[\Phi],\tau_m
 \right\rangle
 =
 [d\nu_M^3/m]\in\Q/\Z.
\end{equation}
In addition, if $d \neq0$, then there exists an injective homomorphism $ \Q/\Z \hookrightarrow H_3^{\rm gr} ( \GG_M )$.
\end{lem}

\begin{proof}
Let \(D_{\GG_M}\) and \(D_{\SU(2)}\) denote the degree-three
differentiation maps for the indicated groups.
By Proposition~\ref{baas}, the identity germ of \(\Phi\) equals $
 \nu_M^3\mathcal F_\sigma(\widetilde\lambda).$
Proposition~\ref{43333} therefore gives the differentiation of $[\Phi] $ is given by 
$
 D_{\GG_M}([\Phi])
 =
 \nu_M^3[(\widetilde\lambda)_e].
$
Naturality of differentiation with respect to the smooth
homomorphism \(L\) yields
\[
 D_{\SU(2)}\bigl(L^*[\Phi]\bigr)
 =
 (dL_e)^*D_{\GG_M}([\Phi])\\
 =
 \nu_M^3[(L^*\widetilde\lambda)_e]=
 d\nu_M^3[(\lambda_{\rm MC})_e].
\]
On the other hand, by Proposition~\ref{43333}, since the local restriction of 
$c_P^{\rm CS}$ equals $\mathcal{F}_{\sigma}(\lambda_{\rm MC} )$, the definition of $c_P^{\rm CS}$ gives
the equality $
 D_{\SU(2)}\bigl([c_P^{\rm CS}]\bigr)
 =
 [(\lambda_{\rm MC})_e].
$
Consequently,
\[
 D_{\SU(2)}
 \bigl(
 L^*[\Phi]-d\nu_M^3[c_P^{\rm CS}]
 \bigr)=0 \in 
H^3_{\rm GF}(\mathfrak{su}(2);\mathbb R)
\]

Under the standard homogeneous--inhomogeneous identification,
\(D_{\SU(2)}\) is the differentiation map of $\SU(2)$. 
Since \(\SU(2)\) is \(2\)-connected, that map is injective (see \cite[Remark~V.14]{WW}).
Thus \eqref{hphp} follows.

Applying the forgetful homomorphism from locally smooth to ordinary
group cohomology gives the corresponding identity in
\(H^3_{\rm gr}(\SU(2);\R/\Z)\).
Finally, for the homomorphism $
 R_m:\mathbb Z/m\rightarrow \SU(2)$
from \eqref{jmjm}, and for the generator
\(\tau_m\in H_3^{\rm gr}(\mathbb Z/m;\mathbb Z)\) used in
Theorem~\ref{tsukau}, one has
\[
 \left\langle
   (L\circ R_m)^*[\Phi],
   \tau_m
 \right\rangle
 =
d \nu_M^3 
 \left\langle
    R_m^*[c_P^{\rm CS}],
   \tau_m
 \right\rangle
 =
 [\, d  \nu_M^3 /m\,]\in\mathbb Q/\mathbb Z .
\]
In addition, suppose $d \neq 0$. 
The compatibility $R_{km}([a])^k=R_m([a])$ is immediate, so the hypothesis of Lemma~\ref{321a}\textup{(ii)} is satisfied. Therefore we obtain an injective homomorphism $ \Q/\Z \hookrightarrow H_3^{\rm gr} ( \GG_M )$.
\end{proof}



\subsection{Proof of Theorem~\ref{main2}}\label{sec3a884}
We prove Theorem~\ref{main2}.
We begin by recalling some facts about volume-preserving diffeomorphism groups of spherical space forms.
Throughout this section, $\Gamma$ denotes one of the finite subgroups of $\SU(2)$ listed in Theorem~\ref{main2}, and we set $M=\SU(2)/\Gamma$. Here, $\SU(2)$ is embedded in $\SO(4)$ by right multiplication on $S^3$.
Let $v$ 
be the quotient of the standard $\SO(4)$-invariant volume form $v_{S^3}$ on $S^3$, satisfying the normalization $\int_{S^3} v_{S^3} =1$. 
Thus the low-dimensional homotopy groups needed below may be read off from the compact Lie group $\Isom^+(M)$.
\begin{fact}[{see, e.g., \cite{BK,HKMR}}]\label{aaaas}
Let $\Gamma$ be one of the finite subgroups of $\SU(2)$ listed in Theorem~\ref{main2} with $\Gamma \neq 1$, and
let $M=S^3/\Gamma$ be the resulting $3$-manifold. 

Then
the action of ${\rm Isom}^+(M)$ on $M$ is transitive. 
If $\Gamma $ is either $D_{4n}$ or the tetrahedral $T_{24}$, then
${\rm Isom}^+(M) \cong \mathrm{O}(3)$
and otherwise ${\rm Isom}^+(M) \cong \mathrm{SO}(3)$.
In particular, 
$\pi_0(\Isom^+(M))$ is either trivial or $ \Z/2$, $\pi_1(\Isom^+_0(M))\cong \Z/2$, $\pi_2(\Isom^+(M))=0$, and $\pi_3 (\Isom^+(M))\cong \Z$.
 \end{fact}

We also discuss the assumption $(\dagger)_q$ by recalling some classical facts and the Smale conjecture.
It is a classical fact that $\Diff(M)$ is a Fr\'echet--Lie group modeled on the Fr\'echet space
$\mathfrak{X}(M)$ of smooth vector fields; see, e.g., \cite{KrieglMichor}.
The subgroup $\Diff_v(M):=\{f\in\Diff(M)\mid f^*v=v\}$ is a Fr\'echet--Lie subgroup.
Consider the map
\begin{equation*}
  \Diff^+(M)\longrightarrow
 \Bigl\{w\in\AAA^{\dim M}(M)\ \Bigm|\ w_p >0 \textrm{ for every } p\in M
, \ \int_M w=\int_M v\Bigr\},
 \qquad f \mapsto 
f^*(v).
\end{equation*}
It is known that this map is a Serre fibration and that its target is convex (hence
contractible); see, for example, \cite{EM}.
Consequently, the inclusion $\Diff_v(M)\hookrightarrow \Diff^+(M)$ is a homotopy equivalence. Combining Ebin--Marsden's fibration with Bamler--Kleiner's solution of the generalized Smale conjecture \cite{BK}, the inclusion
$ \Isom^+(M)\hookrightarrow \Diff_v(M) $
is a homotopy equivalence. 
By Fact~\ref{aaaas}, $\pi_0(\Diff_v(M))$ and $\pi_1(\Diff_{v,0}(M))$ are finite, and $\pi_2(\Diff_{v,0}(M))=0$; equivalently, $(M,v)$ satisfies the assumption $(\dagger)_3$.

Next, we define a $3$-form $\tilde{\lambda}_M$ on $\Diff_v(M)$.
Fix $a_0 \in M$ and consider the evaluation map
\begin{equation}\label{8884}
 {\rm ev}_{a_0} : \Diff_v(M) \longrightarrow M, \qquad f\mapsto f(a_0).
 \end{equation}
We set $ \tilde{\lambda}_M := |\Gamma|\,{\rm ev}_{a_0}^*(v)$. Then $\tilde{\lambda}_M$ is a smooth $3$-form on $\Diff_v(M)$ invariant under left translations.
We start the proof of Theorem~\ref{main2}.
In the proofs below, we often implicitly compose with the comparison map $H_{\rm loc,sm}^3 \ra H_{\rm gr}^3$ when evaluating pairings.
\begin{proof}[Proof of Theorem~\ref{main2}]
For each spherical space form $M=S^3/\Gamma$, let $\GG_M:=\Diff_{v,0}(M)$. 
Choose an identity neighborhood $U\subset \GG_M$ and the associated sequence of \emph{$U$-small tuples}
$S_{n,\GG_M,U}\subset \GG_M^{n+1}$ together with the filling $\sigma$ of Section~\ref{sec223344}.

We verify the period condition in Proposition~\ref{baas}.
Let \(\alpha\in\pi_3(\GG_M)\), and choose a smooth based representative
\(g_\alpha:(S^3,s_0)\rightarrow(\GG_M,e)\).
Let \(p_\Gamma:S^3\to M=S^3/\Gamma\) be the quotient map.  Since
\(p_\Gamma^*v=v_{S^3}\), \(\deg(p_\Gamma)=|\Gamma|\), and
\(\int_{S^3}v_{S^3}=1\), we have
\(\int_Mv=1/|\Gamma|\).  Hence
\[
 \begin{aligned}
 \int_{S^3}g_\alpha^*\widetilde\lambda_M
 &=
 |\Gamma|
 \int_{S^3}({\rm ev}_{a_0}\circ g_\alpha)^*v  \\
 &=
 |\Gamma|\,
 \deg({\rm ev}_{a_0}\circ g_\alpha)\int_Mv \\
 &=
 \deg({\rm ev}_{a_0}\circ g_\alpha)
 \in\mathbb Z.
 \end{aligned}
\]
Thus the period condition in Proposition~\ref{baas} holds.

Let $
i_M:\SO(3)\subset \Isom^+(M)\hookrightarrow \Diff_v(M)$ be the inclusion, and let $q:\SU(2)\ra\SO(3)$ be the covering. 
By Fact~\ref{aaaas}, one has $
\pi_3(\GG_M)\cong \pi_3(\Diff_v(M))\cong \pi_3(\SO(3))\cong \Z,$
and the class $[i_M \circ q]$ generates $\pi_3(\GG_M)$.
Here $\nu_M=|\pi_1(\GG_M)|=2$. 
Then Proposition~\ref{baas} applies with $\Lambda=\Z$ and yields $\Phi_M\in Z^3_{\rm loc,sm}(\GG_M;\R/\Z)$.
Moreover, the evaluation map associated with $i_M\circ q$ is the quotient map $S^3\to S^3/\Gamma$; hence Lemma~\ref{lem:comparison-CCS} applies with $d=|\Gamma|$.
Thus the pullback $(i_M\circ q)^*[\Phi_M]$ is cohomologous to the Chern--Simons 3-class $8|\Gamma|[c_P^{\rm CS}]$. 
Hence we have the required inclusion $\Q/\Z\subset H_3^{\rm gr}(\Diff_{v,0}(M))$.


It remains to pass from $\Diff_{v,0}(M)$ to $\Diff_v(M)$. If $\Gamma$ is either $D_{4n}$ or $T_{24}$,
the nontrivial component acts on the cyclic subgroup $ i_M  \circ q\circ R_m(\Z/m)\subset\SO(3)$ by inversion, and this automorphism acts trivially on $H_3^{\rm gr}(\Z/m) \cong \Z/m$,
since in general the homomorphism $\Z/m \ra \Z/m$ sending $x$ to $x^{n}$ induces 
multiplication by $n^2$ on $H_3^{\rm gr} (\Z/m)  $. 
Hence the detected class in $H_3^{\rm gr}( \Diff_{v,0}(M) )$ is stable under the conjugation action of $ \pi_0( \Diff_v(M))$. 
On the other hand, 
if $\Gamma$ is neither $D_{4n}$ nor $T_{24}$, then $ \Isom^+(M)$ and $\Diff_v(M ) $ are both connected by Fact~\ref{aaaas}.
In summary, Lemma~\ref{321a}\textup{(iii)} applied to
$
G=\Diff_{v,0}(M)\triangleleft\Diff_v(M),$ and $\phi=\Phi_M,$
implies an inclusion $
\Q/\Z\subset H_3^{\rm gr}(\Diff_v(M))$.
This finishes the proof.
\end{proof}

\subsection{Proof of Theorem~\ref{main1}}\label{sec3a8842}

Identify $S^3$ with $\SU(2)$, let $v=\lambda_{\rm MC}$ be the
normalized bi-invariant volume form, and put
$G_v:=\Diff_v(S^3)$ and $\mathfrak g_v:=\mathfrak X_v(S^3)$.
The Ebin--Marsden fibration and the Smale conjecture~\cite{EM,HKMR}
imply that $G_v$ is connected and that the standard inclusion
$\SO(4)\hookrightarrow G_v$ is a homotopy equivalence.  Hence
$\pi_1(G_v)\cong\mathbb Z/2$, $\pi_2(G_v)=0$, and
$\pi_3(G_v)\cong\mathbb Z^2$.

Let $L,R:\SU(2)\to G_v$ be given by $L(a)(x)=ax$ and
$R(a)(x)=xa^{-1}$.  We first construct a second closed left-invariant
$3$-form on $G_v$.

\begin{lem}\label{lem:framing-cocycle-S3}
There exists a continuous Lie algebra $3$-cocycle $\eta_{\rm fr}$ on
$\mathfrak g_v$ such that
\[
 L^*\eta_{\rm fr}=0,
 \qquad
 R^*\eta_{\rm fr}=\lambda_{\rm MC}.
\]
\end{lem}

\begin{proof}
Let $\kappa$ be the left Maurer--Cartan form on $S^3=\SU(2)$.
For $X\in\mathfrak X(S^3)$, define
$\vartheta(X)\in C^\infty(S^3,\operatorname{End}(\mathfrak{su}(2)))$
by $\mathcal L_X\kappa=-\vartheta(X)\kappa$.
This is the crossed homomorphism associated with the left-invariant
trivialization of $TS^3$.  By
Billig--Neeb~\cite[Definition~II.7]{BN},
\[
 \varphi_2(X_1,X_2,X_3)
 :=\sum_{\sigma\in\mathfrak S_3}\!\operatorname{sgn}(\sigma)\,
 \operatorname{tr}\!\bigl(
 \vartheta(X_{\sigma(1)})
 \vartheta(X_{\sigma(2)})
 \vartheta(X_{\sigma(3)})
 \bigr)
\]
is a continuous $C^\infty(S^3)$-valued Lie algebra $3$-cocycle.
For $X\in\mathfrak g_v$ and $f\in C^\infty(S^3)$, one has
$\int_{S^3}X(f)v=0$.  Thus integration against $v$ is a
$\mathfrak g_v$-module homomorphism to the trivial module $\mathbb R$,
and
\[
 \eta_0(X_1,X_2,X_3)
 :=\int_{S^3}\varphi_2(X_1,X_2,X_3)\,v
\]
is a continuous real-valued Lie algebra $3$-cocycle on $\mathfrak g_v$.

The left action preserves $\kappa$, so $L^*\eta_0=0$.
For the right action, $R(a)^*\kappa=\operatorname{Ad}_a\kappa$, and hence
$\vartheta(dR(A))=-\operatorname{ad}_A$.  Cyclicity of the trace gives
\[
 (R^*\eta_0)(A,B,C)
 =-3 \,\operatorname{tr}\!\bigl(
 \operatorname{ad}_A\operatorname{ad}_{[B,C]}\bigr).
\]
Using
$\operatorname{tr}(\operatorname{ad}_A\operatorname{ad}_D)
 =4\operatorname{Tr}(AD)$ and
$\lambda_{\rm MC}(A,B,C)=-(8\pi^2)^{-1}\operatorname{Tr}(A[B,C])$,
we obtain $R^*\eta_0= 96\pi^2\lambda_{\rm MC}$.
Thus $\eta_{\rm fr}:=(96\pi^2)^{-1}\eta_0$ has the required
restrictions.
\end{proof}

The same symbol $\eta_{\rm fr}$ denotes the smooth closed
left-invariant $3$-form on $G_v$ determined by this cocycle.
Let $q_e:G_v\to S^3$ be evaluation at the identity, and put
$\widetilde\lambda_v:=q_e^*v$.
Since $q_e\circ L=\operatorname{id}_{\SU(2)}$ and
$q_e\circ R=\operatorname{inv}$, one has
$L^*\widetilde\lambda_v=\lambda_{\rm MC}$ and
$R^*\widetilde\lambda_v=-\lambda_{\rm MC}$.

The covering $\SU(2)\times\SU(2)\to\SO(4)$ and the homotopy
equivalence $\SO(4)\simeq G_v$ show that $[L]$ and $[R]$ form a basis
of $\pi_3(G_v)$.  The periods of $\widetilde\lambda_v$ and
$\eta_{\rm fr}$ on this basis are $(1,-1)$ and $(0,1)$, respectively.
The periods of both forms on the chosen basis of $\pi_3(G_v)$ are integral. 

Since $\nu_{S^3}=|\pi_1(G_v)|=2$, Proposition~\ref{baas}, applied with
$q=3$ and $\Lambda=\mathbb Z$, gives locally smooth homogeneous group
$3$-cocycles
$\Phi_v,\Phi_{ \rm fr}\in Z^3_{\rm loc,sm}(G_v;\mathbb R/\mathbb Z)$
whose identity germs are $8\mathcal F_\sigma(\widetilde\lambda_v)$ and
$8\mathcal F_\sigma(\eta_{ \rm fr})$, respectively.

Let $\gamma=[c_P^{\rm CS}]\in
H^3_{\rm loc,sm}(\SU(2);\mathbb R/\mathbb Z)$ be the primitive
Cheeger--Chern--Simons class, normalized as in
Theorem~\ref{tsukau}, so that
$\langle R_m^*\gamma,\tau_m\rangle=[1/m]$ for the compatible cyclic
system fixed there.  Lemma~\ref{lem:comparison-CCS} gives
\begin{equation}\label{eq:S3-two-restrictions}
 \begin{aligned}
 L^*[\Phi_v]&=8\gamma,
 &\qquad R^*[\Phi_v]&=-8\gamma,\\
 L^*[\Phi_{\rm fr}]&=0,
 & R^*[\Phi_{\rm fr}]&=8\gamma.
 \end{aligned}
\end{equation}

The compatible systems $L\circ R_m$ and $R\circ R_m$ induce
homomorphisms
$\theta_L,\theta_R:\mathbb Q/\mathbb Z\to
H_3^{\rm gr}(G_v;\mathbb Z)$.  Define
\[
 \Theta_v:(\mathbb Q/\mathbb Z)^2
 \longrightarrow H_3^{\rm gr}(G_v;\mathbb Z),
 \qquad
 \Theta_v(x,y):=\theta_L(x)+\theta_R(y).
\]

\begin{proof}[Proof of Theorem~\ref{main1}]
When taking Kronecker pairings, we implicitly apply the comparison map
from locally smooth to ordinary group cohomology.  By
\eqref{eq:S3-two-restrictions}, pairing with $[\Phi_v]$ and
$[\Phi_v +2\Phi_{\rm fr}]$ sends $\Theta_v(x,y)$ to
$\bigl(8(x-y),8(x+y)\bigr)$. 
This endomorphism of $(\mathbb Q/\mathbb Z)^2$ has determinant $128$
and finite kernel.  Since $\ker\Theta_v$ is contained in that kernel,
it is finite.  Therefore
\[
 \operatorname{Im}\Theta_v
 \cong(\mathbb Q/\mathbb Z)^2/\ker\Theta_v
 \cong(\mathbb Q/\mathbb Z)^2.
\]
Consequently,
$H_3^{\rm gr}(\Diff_v(S^3);\mathbb Z)$ contains a subgroup isomorphic
to $(\mathbb Q/\mathbb Z)^2$.
This proves Theorem~\ref{main1}.
\end{proof}
The precise source of the two detected
\(\mathbb Q/\mathbb Z\)-factors is the restriction calculation
\eqref{eq:S3-two-restrictions}: the two finite-cyclic evaluations give
an integer pairing matrix of nonzero determinant.  The proof does not
require an identification of the two cocycles with transgressions of
the Euler and Pontryagin classes, which lie in $H^4(BSO(4);\mathbb{Z}) \cong \mathbb{Z} \oplus \mathbb{Z}$.


Next, for the contact structure on $S^3$, we obtain a corollary.
Here, $\alpha_{\rm st}$ is the standard contact 1-form on $S^3$,
that coincides with a multiple of the Maurer-Cartan form $\omega_{\rm MC}$.
In particular, $\alpha_{\mathrm{st}}$ is strictly invariant under the left $SU(2)$-action.
: 
\begin{cor}\label{main1cor}
Let $M:=S^3$, and $\lambda:= \alpha_{\rm st}$. 
Then $H_3^{\rm gr}\bigl(\Diff_{\alpha_{\rm st}}(S^3)\bigr)$ contains a subgroup isomorphic to $\Q/\Z \oplus \Q/\Z$.
\end{cor}
\begin{proof}   
Let
$Q:=\Diff_{\alpha_{\rm st}}(S^3)$.
Since $ \alpha_{\rm st} \wedge d \alpha_{\rm st}$ is a multiple of $v$,
we have the inclusion $i:Q\hookrightarrow G_v$.
Since 
The left $\SU(2)$-action preserves $\alpha_{\rm st}$,
and $R\circ R_m$ lies in the Reeb circle.  Thus $R\circ R_m$ factors through $Q$.  Hence, the same compatible
systems induce a homomorphism
$\Theta_Q:(\mathbb Q/\mathbb Z)^2\to H_3^{\rm gr}(Q;\mathbb Z)$.
Naturality gives $i_*\circ\Theta_Q=\Theta_v$, and hence
$\ker\Theta_Q\subset\ker\Theta_v$.  The preceding argument
shows that $\ker\Theta_Q$ is finite.  Thus
$H_3^{\rm gr}(\Diff_{\alpha_{\rm st}}(S^3);\mathbb Z)$ contains a
subgroup isomorphic to $(\mathbb Q/\mathbb Z)^2$ as required.
\end{proof}

\subsection{The hyperbolic case: a construction and a question}\label{hyp11}
As a brief digression, since the present paper concerns $3$-dimensional geometry, it is worth commenting on the hyperbolic case.
We briefly explain how the same method yields group $3$-cocycles in the hyperbolic case.

Let $M$ be a closed hyperbolic $3$-manifold, and let $v$ be its hyperbolic volume form.
Fix a base point $a_0\in M$, and consider the pullback $ \tilde{\lambda}_M:={\rm ev}_{a_0}^*(v)$ under the evaluation map as in \eqref{8884}. 
Then $\tilde{\lambda}_M$ is a smooth left-invariant $3$-form on $\Diff_v(M)$.

Let us review some results.
By Mostow--Prasad rigidity, $\Isom^+(M)$ is finite.
Moreover, combining Ebin--Marsden's fibration with Gabai's theorem \cite{Gabai},
the inclusion $
\Isom^+(M)\hookrightarrow \Diff_v(M)$
is a homotopy equivalence.
Hence $
\Diff_{v,0}(M)$ is contractible, and $
\pi_0(\Diff_v(M))\cong \Isom^+(M)$
is a finite group.

Therefore, in degree $3$ there is no period obstruction for the construction of
Section~\ref{secDiff0Configured}.
Denote $\Diff_{v,0}(M)$ by $G$.
Choose an identity neighborhood $U\subset G$ and the corresponding sequence of
$U$-small tuples $S_{n,G,U}$ together with a filling $\sigma$ 
as in
Section~\ref{sec223344}.
Applying the construction of Section~\ref{secDiff0Configured} with $\Lambda=\{0\}$
to the form $\tilde{\lambda}_M|_{G}$ yields a locally smooth $\R$-valued group $3$-cocycle
$
\Phi_M\in Z^3_{\rm loc,sm}(G;\R).
$
Since $\pi_0(\Diff_v(M))\cong \Isom^+(M)$ is finite and the coefficient group is $\R$,
as in Section~\ref{sec3a4}, we obtain the transfer 
$\mathrm{cor}^{\Diff_v(M)}_G(\Phi_M)$, which is locally smooth by Lemma \ref{321a} (i) .
At present we do not know whether the 3-class $[\mathrm{cor}^{\Diff_v(M)}_G (\Phi_M)] $ is nonzero
in $ H^3_{\rm gr}(\Diff_v(M);\R)$.
We pose a question for future work:

\medskip\noindent
\textbf{Question.}
Is the transferred class of some such extension nonzero?
More concretely, does there exist a group-homology $3$-class %
$\kappa \in H_3^{\rm gr}(\Diff_v(M);\Z)$
such that $
\langle \mathrm{cor}(\Phi_M),\kappa\rangle\neq0$?
A positive answer would imply that $H_3^{\rm gr}(\Diff_v(M))$
contains a subgroup isomorphic to $\Z$.


\section{Configured $3$-cocycles on groups of symplectomorphisms}\label{sec32884}

Throughout, let $(M,\omega)$ be a closed $2n$-dimensional symplectic manifold, and let
$\Diff_{\omega}(M)$ be the group of symplectomorphisms of $(M,\omega)$, with identity
component $\Diff_{\omega,0}(M)$.
In Section~\ref{sec3ss2884}, we construct a natural
\(\Diff_{\omega}(M)\)-invariant Lie algebra \(3\)-cocycle and explain
how, under the hypotheses used in the applications,
Proposition~\ref{baas} produces a locally smooth group \(3\)-cocycle
on \(\Diff_{\omega,0}(M)\).  In
Section~\ref{sec3sdds2884}, we prove Theorem~\ref{main3}.

\subsection{Invariant $3$-forms on $\Diff_{\omega}(M)$}\label{sec3ss2884}

We recall the basic Hamiltonian formalism (see, e.g., \cite{Can}).
Given $f\in C^\infty(M)$, the corresponding Hamiltonian vector field $X_f\in\mathfrak{X}(M)$
is defined by $\iota_{X_f}\omega=-df$,
where $\iota$ denotes interior multiplication.
The Poisson bracket of $f,g\in C^\infty(M)$ is given by
$\{f,g\}:=\omega(X_f,X_g)=X_f(g)\in C^\infty(M)$.
With this bracket $C^\infty(M)$ becomes a (topological) Lie algebra.
Let $\R_{\rm const}$ be the subalgebra consisting of constant functions.
We set
\[
  \mathfrak{ham}(M,\omega)
  := \{ X_f \in \mathfrak{X}(M) \mid \iota_{X_f}\omega=-df,\ f\in C^\infty(M)\}
\]
for the Lie algebra of Hamiltonian vector fields.
Then the correspondence $
 f\ra X_f $
induces a Lie algebra isomorphism $C^\infty(M)/\R_{\rm const} \cong \mathfrak{ham}(M,\omega)$.

Next, we recall the known exact sequence \eqref{fffo} and prove Lemma~\ref{asax}. 
By Cartan's formula, $  L_X\omega = d(\iota_X\omega)$
for every $X\in\mathfrak{X}(M)$.
For $X\in\mathfrak{X}(M)$ with $ L_X\omega=0 $, put
$  \kappa  (X) := [\iota_X\omega] \in H^1(M;\mathbb{R}).
$
Then
\begin{equation}\label{fffo}
  0 \longrightarrow
  \mathfrak{ham}(M,\omega)
  \longrightarrow
  \{ X\in\mathfrak{X}(M) \mid L_X\omega=0\}
  \stackrel{\kappa}{\longrightarrow}
  H^1(M;\mathbb{R})
  \longrightarrow 0
\end{equation}
is exact.
We next record an elementary identity, which is standard.
\begin{lem}\label{asax}
For $f,g\in C^\infty(M)$, define $\langle f,g\rangle_M:=\int_M fg\,\omega^n$.
Then for all $f,g,h\in C^\infty(M)$,
\begin{equation}\label{eq:ad-invariance-lemma}
  \langle \{f,g\},h\rangle_M
  + \langle g,\{f,h\}\rangle_M
  = 0.
\end{equation}
\end{lem}

\begin{proof}
Liouville's theorem gives $L_{X_f}\omega^n=0$ for every $f$.
Hence, for any $\varphi\in C^\infty(M)$,
\[
  L_{X_f}(\varphi\,\omega^n )
   = X_f(\varphi)\,\omega^n  + \varphi\,L_{X_f}\omega^n
   = X_f(\varphi)\,\omega^n.
\]
By Stokes' theorem and the compactness of $M$, we have
\begin{equation}\label{dddd}
 \int_M X_f(\varphi)\,\omega^n
   = \int_M L_{X_f}(\varphi\,\omega^n )
   = \int_M d\bigl(\iota_{X_f}(\varphi\,\omega^n )\bigr)
   = 0.
\end{equation}
We prove \eqref{eq:ad-invariance-lemma} by computing
\[
\begin{aligned}
  \langle \{f,g\},h\rangle_M
  + \langle g,\{f,h\}\rangle_M
  &= \int_M \{f,g\}\,h\,\omega^n
     + \int_M g\,\{f,h\}\,\omega^n  \\
  &= \int_M \bigl(X_f(g)\,h + g\,X_f(h)\bigr)\,\omega^n = \int_M X_f(gh)\,\omega^n = 0,
\end{aligned}
\]
where we used \eqref{dddd} with $\varphi=gh$.
\end{proof}

In general, for a quadratic Lie algebra $(\mathfrak{g},[\ ,\ ],\langle\ ,\ \rangle)$, one defines the
alternating trilinear form
\begin{equation}\label{phiphi}
  \Phi(x,y,z) := \langle x,[y,z]\rangle,\qquad x,y,z\in\mathfrak{g}.\end{equation}
It is well known that $\Phi\in C^3_{\rm GF}(\mathfrak g;\R)$ is a Lie algebra
$3$-cocycle, called the {\it Cartan $3$-cocycle}; see, e.g., \cite[\S 4]{FSS}.
In particular, when $\g$ is a finite-dimensional simple Lie algebra and
$\langle\ ,\ \rangle$ is the Killing form $B$, $H^3_{\rm GF}(\g;\R)\cong\R$ is generated by
the class of $\Phi$.

We adapt this construction to the symplectic situation.
\begin{defn}\label{def:betaM}
Let $(M,\omega)$ be a closed symplectic manifold of dimension $2n$.
Put
\[
C_0^\infty(M):=
\{f\in C^\infty(M)\ \mid \
\int_M f\,\omega^n=0 \}.
\]
Every Hamiltonian vector field has a unique Hamiltonian in
$C_0^\infty(M)$, and the map $f\mapsto X_f$ is a Lie algebra
isomorphism $C_0^\infty(M)\cong\mathfrak{ham}(M,\omega)$.
For normalized Hamiltonians $f,g,h$, define
\[
\beta_M(X_f,X_g,X_h)
:=A_n\int_M f\{g,h\}\,\omega^n,
\qquad
A_n:=\frac{(n+1)(n+2)}{2^{n+3}\pi^{n+2}}.
\]
The coefficient $A_n$
is chosen so that the restriction to the standard \(\SU(2)\)-subgroup of
\(\mathrm{PSU}(n+1)\) agrees with \(\lambda_{\rm MC}\); see
Lemma~\ref{ddddd}.
\end{defn}

By construction, $\beta_M$ is smooth and invariant under the diagonal action of
$\Diff_{\omega}(M)$ on $\mathfrak{ham}(M,\omega)$, and Lemma~\ref{asax} and \eqref{phiphi} show that
$\beta_M$ is a Lie algebra $3$-cocycle. Hence $\beta_M$ defines a cohomology class $ [\beta_M]\in H^3_{\rm GF}(\mathfrak{ham}(M,\omega);\R). $
In particular, if $H^1(M;\R)=0$, then \eqref{fffo} yields the identification
$ \mathfrak{ham}(M,\omega)
 =  \{ X\in\mathfrak{X}(M) \mid L_X\omega=0\},$
the full Lie algebra of symplectic vector fields. In this case we may regard $\beta_M$ as a left-invariant closed $3$-form on the symplectic diffeomorphism group $\Diff_{\omega}(M)$.

\subsection{Proof of Theorem~\ref{main3}}\label{sec3sdds2884}

We specialize to projective spaces and prove Theorem~\ref{main3}.

We first prove Lemma~\ref{ddddd}. Let $M=\C P^n$ and let $\omega_n$ be the Fubini--Study symplectic form with normalization $\int_{\C P^1}\omega_n =2\pi $.
Since the action of ${\rm PSU}(n+1)$ on $\C P^n$ preserves $\omega_n$, we have an inclusion $i_n: {\rm PSU}(n+1)\hookrightarrow \Diff_{\omega_n}(\C P^n)$.
We begin by understanding the restriction of $\beta_M$ to ${\rm PSU}(n+1) $.


\begin{lem}\label{ddddd}
Let \(\mathbb CP^n\) be equipped with the Fubini--Study symplectic form
\(\omega_n\), normalized by \(\int_{\mathbb CP^1}\omega_n=2\pi\).  Let
\(i_n:\mathrm{PSU}(n+1)\hookrightarrow \Diff_{\omega_n}(\mathbb CP^n)\)
be the standard inclusion, let
\(j_n:\mathfrak{su}(2)\hookrightarrow\mathfrak{psu}(n+1)\) be the canonical
upper-left block inclusion, and let
\(q:\SU(2)\to\SO(3)=\mathrm{PSU}(2)\) be the double covering.  

Then, under the identification
\(C^3_{\rm GF}(\mathfrak{su}(2);\mathbb R)
 \cong \Omega^3(\SU(2))^{\SU(2)}\),
the following equality
holds in \(C^3_{\rm GF}(\mathfrak{su}(2);\mathbb R)\):
\[
 (d i_n \circ j_n)^*(\beta_{\mathbb CP^n})
 =
 q^*i_1^*(\beta_{\mathbb CP^1})=\lambda_{\rm MC}. 
\]
Moreover, for \(M=\mathbb CP^1\times\mathbb CP^1\) with
\(\omega=\omega_1\oplus\omega_1\), and for either factor inclusion
\(\iota_k:\SO(3)\hookrightarrow\Diff_{\omega}(M)\), \(k=1,2\), one has
$q^*\iota_k^*(\beta_M)=4\,q^*i_1^*(\beta_{\mathbb CP^1})$. 


Moreover, 
the restriction \(i_n^*(\beta_{\mathbb CP^n})\) represents a nonzero class in
\(H^3_{\rm GF}(\mathfrak{psu}(n+1);\mathbb R)\).
\end{lem}
\begin{proof}
Let \(X,Y,Z\) be the standard basis of \(\mathfrak{su}(2)\), chosen with the
same sign convention as in \eqref{mc3}, and let \(x,y,z\) be the corresponding
moment Hamiltonians for the upper-left \(\SU(2)\)-action on \(\mathbb CP^n\).
Thus \(\{x,y\}=z\), \(\{y,z\}=x\), and \(\{z,x\}=y\).  Since
\((\Omega^3\mathfrak{su}(2)^*)^{\SU(2)}\) is one-dimensional, it is enough
to evaluate on \(X,Y,Z\).

Write \(\mu_i([Z_0:\cdots:Z_n])=|Z_i|^2/\sum_{r: 0 \leq r \leq n}|Z_r|^2\).  We may take
\(x=(\mu_0-\mu_1)/2\).  The standard Dirichlet integral for the normalized
Fubini--Study measure gives
\[
 \int_{\mathbb CP^n}x^2\,\omega_n^n
 =
 (2\pi)^n\frac{1}{2(n+1)(n+2)}
 =
 \frac{2^{n-1}\pi^n}{(n+1)(n+2)}.
\]
Therefore
\[
( d i_n \circ  j_n)^*(\beta_{\mathbb CP^n})(X,Y,Z)
 =
 A_n\int_{\mathbb CP^n}x\{y,z\}\,\omega_n^n
 =
 \frac{1}{16\pi^2}.
\]
By the normalization of \(\lambda_{\rm MC}\), we also have
\(\lambda_{\rm MC}(X,Y,Z)=1/(16\pi^2)\).  Hence
\(( d i_n \circ  j_n)^*(\beta_{\mathbb CP^n})=\lambda_{\rm MC}\).  Taking \(n=1\)
gives \(q^*i_1^*(\beta_{\mathbb CP^1})=\lambda_{\rm MC}\), and therefore
the first equality follows.

For the product case, let \(x,y,z\) be pulled back from the \(k\)-th factor.
Since \((\omega_1\oplus\omega_1)^2
=2\,\mathrm{pr}_1^*\omega_1\wedge\mathrm{pr}_2^*\omega_1\) and
\(\int_{\mathbb CP^1}\omega_1=2\pi\), we get
\[
 \int_{\mathbb CP^1\times\mathbb CP^1}
 x\{y,z\}\,(\omega_1\oplus\omega_1)^2
 =
 4\pi\int_{\mathbb CP^1}x\{y,z\}\,\omega_1.
\]
Using \(A_1=3/(8\pi^3)\) and \(A_2=3/(8\pi^4)\), this gives
\(q^*\iota_k^*(\beta_M)=4\,q^*i_1^*(\beta_{\mathbb CP^1})\).

Finally, \( (d i_n)^*(\beta_{\mathbb CP^n})\) is nonzero in cohomology, because
its pullback along \(j_n\) is \(\lambda_{\rm MC}\), whose class is nonzero in
\(H^3_{\rm GF}(\mathfrak{su}(2);\mathbb R)\).
\end{proof}

We recall two facts about symplectomorphism groups.
First, when $M=\C P^n$ and $n=1,2$, the inclusion $
 {\rm PSU}(n+1)\subset \Diff_{\omega_n}(\C P^n)$
is a homotopy equivalence: for $n=1$ this is the classical Smale conjecture, and for
$n=2$ it is due to Gromov \cite{Gro}. 
Since for $n>2$ no analogous result is presently known, we use only the cases \(n=1,2\). Here we only record that $\pi_1({\rm PSU}(n+1))\cong \Z/(n+1) $.

Second, for the product 
$(\C P^1 \times\C P^1 ,\omega_1 \oplus\omega_1)$ there is a homotopy equivalence
\begin{equation}\label{homotopy1}
\bigl(\SO(3)\times \SO(3)\bigr)\rtimes \mathbb{Z}/2 
 \;\simeq\; \Diff_{\omega_1 \oplus\omega_1}(S^2\times S^2) \cong \Diff_{\omega_1 \oplus\omega_1,0}(S^2\times S^2) \rtimes \Z/2
\end{equation}
(see \cite{AM}). Here, the left-hand homotopy equivalence is induced by the inclusion.
\begin{proof}[Proof of Theorem~\ref{main3}] 
Let $M$ be one of $\C P^1$, $\C P^2$, or $\C P^1\times \C P^1$.
Put $\Gamma_M=\Diff_{\omega}(M)$ and $\GG_M=\Diff_{\omega,0}(M)$.
In all three cases $\pi_0(\Gamma_M)$ is finite, $\pi_2(\GG_M)=0$, and $\nu_M:=|\pi_1(\GG_M)|<\infty$; hence $(M,\omega)$ satisfies $(\dagger)_3$.
Fix an identity neighborhood $U\subset \GG_M$ and a filling $\sigma$ supplied by Proposition~\ref{prop:local-affine-filling}.

\smallskip\noindent
\textbf{Case 1: $M=\C P^n$ with $n=1,2$.} 
Consider the composite 
\begin{equation}\label{gggg} 
\SU(2)\subset \SU(n+1) \twoheadrightarrow {\rm PSU}(n+1)\subset \Diff_{\omega_n}(\C P^n),
\qquad n=1,2.
\end{equation}
Here the first map is the canonical upper-left block inclusion; denote the composite by $\iota_n$.
Then $[\iota_n]$ is a generator of $\pi_3(\GG_M)\cong \Z$ up to sign, by the homotopy equivalence above and Bott periodicity.
By Lemma~\ref{ddddd}, $\iota_n^*(\beta_{\C P^n})=\lambda_{\rm MC}$; hence the period condition in Proposition~\ref{baas} holds with $\lambda=\beta_{\C P^n}$ and $\Lambda=\Z$.
Let $\Phi_M$ be the locally smooth cocycle obtained from Proposition~\ref{baas}. Lemma~\ref{lem:comparison-CCS} applies to $\iota_n$ with $d=1$.
Thus we obtain an injection $\Q/\Z\hookrightarrow H_3^{\rm gr}(\GG_M)$; since $\Gamma_M$ is connected in this case, this proves the assertion for $\C P^n$.

\smallskip\noindent
\textbf{Case 2: $M=\C P^1\times \C P^1$ and $\omega=\omega_1\oplus\omega_1$.}
Let $
\tilde\iota_1,\tilde\iota_2:\SU(2)\ra \GG_M
$
be the composites of the double covering $q:\SU(2)\to\SO(3)$ with the two factor inclusions
\[
\SO(3)\hookrightarrow \SO(3)\times \SO(3)\subset \GG_M=\Diff_{\omega,0}(M).
\]
Then the classes $[\tilde\iota_1],[\tilde\iota_2]$ form a basis of $\pi_3(\GG_M)\cong \Z\oplus \Z$.
Moreover, by Lemma~\ref{ddddd}, for $k\in\{1,2\}$ one has
$\tilde\iota_k^*(\beta_M)=4\,q^*i_1^*(\beta_{\C P^1})\in C_{\rm GF}^3(\mathfrak{su}(2);\R)$.
Set $\lambda:=\beta_M/4$. Then the periods on the basis $[\tilde\iota_1],[\tilde\iota_2]$ are integral, so Proposition~\ref{baas} applies with $\Lambda=\Z$.
Let $\tilde\iota_\Delta(a):=\tilde\iota_1(a)\tilde\iota_2(a)$. Then $\tilde\iota_\Delta^*\lambda=2\lambda_{\rm MC}$: after dividing $\beta_M$ by $4$, each factor contributes $\lambda_{\rm MC}$, while the mixed terms vanish because the moment Hamiltonians have zero mean on the other factor. Hence Lemma~\ref{lem:comparison-CCS} applies with $d=2$ and $\nu_M=4$.
Thus we obtain an injection $\Q/\Z\hookrightarrow H_3^{\rm gr}(\GG_M)$.

Next, we check the hypothesis of the transfer lemma, Lemma~\ref{321a}\textup{(iii)}.
The factor-swap sends $\tilde\iota_1$ to $\tilde\iota_2$ and fixes the diagonal class $(\tilde\iota_\Delta\circ R_m)_*(\tau_m)$; hence Lemma~\ref{321a}\textup{(iii)}
implies
$\Q/\Z\subset H_3^{\rm gr}( \Diff_{ \omega_1 \oplus \omega_1} (M))$.
\end{proof}

As a consequence of Case~1 in the proof of Theorem~\ref{main3},
we obtain the following extension of the Cheeger--Chern--Simons $3$-cocycle.
\begin{cor}\label{asaaa}
Let $M=\C P^n$ with the Fubini--Study symplectic form $\omega_n$, and assume $n=1$ or $n=2$.
Let $q_n: {\rm SU}(n+1)\ra {\rm PSU}(n+1) $ be the projection. 
Then there exists a locally smooth 3-cocycle $
\Phi_M\in Z^3_{\rm loc,sm}(\Diff_{\omega_n}(M);\R/\Z)$ such that the pullback along the composite map $i_n \circ q_n : {\rm SU}(n+1)\ra \Diff_{\omega_n}(M) $
is cohomologous to the $(n+1)^3$-multiple of the CCS 3-class. 
\end{cor}
\begin{proof}   
Let \(c_{\rm CCS}\) denote the primitive second
Cheeger--Chern--Simons class on \(\SU(n+1)\). The two classes being
compared have the same differentiated invariant \(3\)-form: their
restrictions to the upper-left \(\mathfrak{su}(2)\) agree, and
\(H^3_{\rm GF}(\mathfrak{su}(n+1);\mathbb R)\) is one-dimensional.
Since \(\SU(n+1)\) is \(2\)-connected, differentiation in degree
three is injective. Hence, the two locally smooth classes are equal.
\end{proof}
\begin{rem}
Reznikov~\cite[\S6]{Rez2} states related Chern--Simons extension
results for \(S^2\), \(\mathbb CP^2\), and 
\(S^2\times S^2\) with some symplectic areas.  The argument above
is independent and gives locally smooth cocycle representatives with
explicit finite-cyclic evaluations.  Its conclusion is the
homological embedding of Theorem~\ref{main3}. 
We do not claim that the two characters coincide or that our
character satisfies the \(K_3(\mathbb C)\)-compatibility considered
there.
\end{rem}




\subsection{Applications to the Quillen plus construction of $\Diff$}
\label{sec323d}

From the viewpoint of algebraic \(K\)-groups and related topics, it is natural to ask whether the discussion above can be lifted to the third homotopy group of Quillen's plus construction.

We first recall Quillen's plus construction (see, e.g., \cite{Wei}).  Let \(\Gamma\) be a group and let
\(P\lhd \Gamma\) be a perfect normal subgroup.  Let \(B\Gamma\) be an
Eilenberg--MacLane space of type \((\Gamma,1)\).  Then there exists a connected
CW-complex \(B\Gamma^+_P\), called the Quillen plus construction, together
with a map \(\eta:B\Gamma\to B\Gamma^+_P\), such that
\(\pi_1(B\Gamma^+_P)\cong \Gamma/P\), and, for every local coefficient system
\(\mathcal L\) on \(B\Gamma^+_P\), the pushforward
\(\eta_*:H_*(B\Gamma;\eta^*\mathcal L)\rightarrow H_*(B\Gamma^+_P;\mathcal L)\)
is an isomorphism.  In the context of algebraic \(K\)-groups, the study of
the homotopy groups of \(B\Gamma^+_P\) is of fundamental importance.

Let \(G\) be one of the diffeomorphism groups appearing in
Theorems~\ref{main1}--\ref{main3}, and let \(G_0\) be its identity
component.  In the plus construction below, \(G\) and \(G_0\) are regarded as
discrete groups.  The subgroup \(G_0\) is normal in \(G\).

For the symplectic cases, \(H^1(M;\mathbb R)=0\), so the symplectic flux
homomorphism vanishes on the identity component.  For the volume-preserving
spherical space form cases, \(H^2(M;\mathbb R)=0\), so the volume flux
homomorphism vanishes.  In both cases, Banyaga's simplicity theorem
\cite[Theorems 4.3.1 and 5.1.3]{Ban} implies that \(G_0\) is perfect.
Thus we can define the plus construction \(BG^+_{G_0}:=B(G^\delta)^+_{G_0^\delta}\).
\begin{thm}\label{thm6}
Let $G$ be one of the diffeomorphism groups appearing in Theorems~\ref{main1}--\ref{main3}, and let $G_0$ be its identity component.
Then
the third homotopy group $\pi_3( BG^+_{G_0})$ contains a subgroup isomorphic to $\Q/\Z$. 
\end{thm}
The proof is based on the following proposition.
Indeed, in the proofs of Theorems~\ref{main1}--\ref{main3},
we constructed the required cocycles and compatible cyclic subgroups
satisfying Proposition~\ref{af8}. Here, if $G$ is not connected, we use the transferred 3-cocycle $\mathrm{cor}^{G}_{G_0}(\Phi)$. 

\begin{prop}\label{af8}
Let $P \lhd \Gamma$ be a perfect normal subgroup together with a homomorphism $\rho:\SU(2)\ra P$.
Fix direct-limit compatible generators $\{ \tau_m \in H_3^{\rm gr}(\Z/m ;\Z)\}_{m \geq 2}$. 
Suppose there exist a nonzero $c_\Gamma \in \N$ and an $\R/\Z$-valued $3$-cocycle $\phi$ of $ \Gamma $ such that 
the rotation homomorphism $R_m : \Z/m \ra  \SU(2)$ in \eqref{jmjm} satisfies 
\[
\bigl\langle \phi|_P,(\rho\circ R_m)_*\tau_m\bigr\rangle
=[c_\Gamma/m]\in\Q/\Z\qquad(m\geq2).
\]
Then the third homotopy group $\pi_3(B\Gamma ^+_P)$ contains a subgroup isomorphic to $\Q/\Z$, and 
there is a homomorphism
$
 \mathcal R_\phi:
 \pi_3(B\Gamma_P^+)\ra \R/\Z
$
whose image contains a subgroup isomorphic to $\Q/\Z\subset\R/\Z$. 
\end{prop}
\begin{proof}
Let $q:\widehat P\to P$ be the universal central extension of $P$.
Since the isomorphism $\kappa: \pi_3(B\Gamma_P^+)\cong H_3^{\rm gr}(\widehat P)$ 
is known (see \cite[Exercises~IV.1.8--IV.1.9]{Wei}),
we will construct an injection $\Q/\Z\hookrightarrow H_3^{\rm gr}(\widehat P)$.

Put $S:=\SU(2)$ and $C := H_2^{\rm gr}(S;\Z)$ for simplicity,
and let $p_S: \widehat S \ra S $ be the universal central extension with kernel $C$.   
By assumption, 
the homomorphisms $\{ R_m\}_{m \geq 2}$ induce a homomorphism $\iota: \Q/\Z\to S$.
Here
$C \cong H_2^{\rm gr}(S;\Z)$ is divisible; see \cite{DPS}.  The pullback
of this extension along $\iota$ has class in $H^2_{\rm gr}(\Q/\Z ;C)$.
Since $C$ is divisible, it is injective as a $\Z$-module, and since
$\Q/\Z  $ is locally cyclic, $ H_2^{\rm gr}(\Q/\Z )=\bigwedge^2(\Q/\Z) =0$.  The universal coefficient theorem immediately implies $H^2_{\rm gr}(\Q/\Z;C)=0$
because of $ \Ext^1_{\Z}(\Q/\Z ;C)=0$.
Hence we have a lift $\widehat\iota: \Q/\Z  \ra  \widehat S$ of $\iota$.

Functoriality of universal central extensions gives
$\widehat\rho:\widehat S\to\widehat P$ with
$q\circ\widehat\rho=\rho\circ p_S$.  
Let $k_m:  \Z/m \hookrightarrow \Q/\Z$ be the injection sending $a$ to $a/m$, and 
$\iota_m :=\widehat\rho\circ\widehat\iota \circ k_m :\Z/m \to\widehat P$. Then 
\[
 \langle   q^* (\phi |_P) ,  ( \iota_m)_* (\tau_m)\rangle =  \langle \phi |_P,  ( q \circ \widehat\rho\circ\widehat\iota \circ k_m )_* (\tau_m)\rangle  
=  
\langle \phi|_P,  (\rho \circ R_m)_* (\tau_m)\rangle  =[c_\Gamma/m]. \] 
By applying Lemma~\ref{321a}\textup{(ii)} to $G= \widehat{P}$ and $ \mathcal{F} = q^* ( \phi |_P)$, we obtain an inclusion $\Q/\Z$ into $H_3^{\rm gr}( \widehat{P};\Z)$, as required.
Finally, the composite of $\kappa$ and the evaluation $ \langle [q^* \phi |_P], \bullet\rangle :
H_3^{\rm gr}(\widehat P ) \ra \R/\Z$ gives the required homomorphism $\mathcal{R}_{\phi}$. 
\end{proof}

\begin{rem}[Relation to Reznikov's regulator]
Cheeger--Chern--Simons cocycles for groups such as \(\SL_n(\mathbb C)\) and
\(\SU(n)\) are closely related to regulator classes in algebraic
\(K\)-theory; see, for example, \cite{DK,DZ,DPS}.  In this sense, the
cocycles considered by Reznikov \cite{Rez1,Rez2} and those in the present paper may be
viewed as Chern--Simons-type \(3\)-cocycles.  

Reznikov~\cite[Theorem~2.2]{Rez1} formulated a regulator from the plus-construction homotopy group $\pi_3( B \Diff_{\omega_2}(M)^+_P) \ra \R/\Z$ with $M=\C P^2 $, together with a compatibility statement involving the Beilinson--Karoubi regulator on
$K_3(\mathbb C)$.  The character above is obtained independently
from the locally smooth group $3$-cocycle constructed in this paper,
and its nontriviality is detected directly on finite cyclic
subgroups.  We do not claim that
this character agrees with Reznikov's regulator, or that it satisfies
the \(K_3(\mathbb C)\)-compatibility considered in~\cite{Rez1}.
\end{rem}

\subsection{Locally smooth non-extension to \(\Diff^+(S^n)\)}\label{Diff+(S3)}
We study whether the locally smooth classes constructed on the
volume-preserving subgroup extend across the inclusion into the full
orientation-preserving diffeomorphism group. We prove the rank-two
non-extension statement for \(S^3\) below. The case of \(S^2\) requires
additional comparison results and is not used in the main theorems

Let \(G^+:=\Diff^+(S^n)\) be the diffeomorphism group of $S^n$ preserving orientation, and let \(\iota:G_v \hookrightarrow G^+\) be
the inclusion, where $v$ is the standard volume $n$-form.
We will show a proposition in the case $n=2$.
\begin{prop}\label{prop:rank-two-nonextension}
Let $n=3$. 
No nonzero integral linear combination of \([\Phi_v]\) and
\([\Phi_{\rm fr}]\) belongs to the image of
\[
 \iota^*:H^3_{\rm loc,sm}(G^+;\R/\Z)
 \longrightarrow H^3_{\rm loc,sm}(G_v;\R/\Z).
\]
\end{prop}

\begin{proof}
Billig--Neeb's formulation of Haefliger's theorem gives
\(H^3_c(\mathfrak{X}(S^3);\R)=0\)
\cite[Corollary~D.5]{BN}.  Hence, by naturality of differentiation, every
class in the image of \(\iota^*\) has zero differentiated class on
\(\mathfrak g_v\).  On the other hand, with the standard normalization,
\[
 D_{G_v}([\Phi_v])=8[\tilde{\lambda}_v],
 \qquad
 D_{G_v}([\Phi_{\rm fr}])=8[\eta_{\rm fr}].
\]
The two Lie algebra classes are linearly independent.  Indeed, their
restrictions to the left and right \(\mathfrak{su}(2)\)-subalgebras are
respectively
\((\lambda_{\rm MC},-\lambda_{\rm MC})\) and
\((0,\lambda_{\rm MC})\).  Therefore, no nonzero integral combination of the
two group classes can lie in the image of \(\iota^*\).
\end{proof}

\begin{rem}\label{rem:diffplus-discrete}
The preceding proposition is a statement about locally smooth cohomology.
It does not determine whether the homomorphism
\(\iota_*\Theta_v:(\Q/\Z)^2\to H_3^{\rm gr}(\Diff^+(S^3);\Z)\) has finite
kernel.  The inclusion \(G_v\hookrightarrow G^+\) is a homotopy equivalence
for the usual topologies, but this gives no injectivity statement for the
homology of the underlying discrete groups.  Thus the present argument does
not prove that \(H_3^{\rm gr}(\Diff^+(S^3);\Z)\) contains
\((\Q/\Z)^2\).
\end{rem}
\begin{rem}[Sharpness for \(S^2\)]
Next, let
\(\iota:\Diff_{\omega_1}(S^2)\to\Diff^+(S^2)\)
denote the inclusion.
Recall the locally smooth 3-cocycle $\Phi_{\mathbb CP^1}$ of $\Diff_{\omega_1}(S^2) $ in Corollary \ref{asaaa}
We will show that no nonzero integral multiple of
\([\Phi_{\mathbb CP^1}]\) lies in the image of
\[
 \iota^*:
 H^3_{\rm loc,sm}(\Diff^+(S^2);\R/\Z)
 \longrightarrow
 H^3_{\rm loc,sm}(\Diff_{\omega_1}(S^2);\R/\Z).
\]

By the coefficient exact sequence and the comparison results of
\cite{WW}, the nonvanishing of \(p_1\) in the continuous
cohomology of \(\Diff^+(S^2)\), proved in
\cite[Theorem~1.2 and diagram~\textup{(1.6)}]{Prigge},
implies that every locally continuous
\(\R/\Z\)-valued degree-three class on \(\Diff^+(S^2)\)
lifts to a globally continuous real class.
Its pullback to the compact rotation group vanishes by averaging.

The pullback of such a real class to the compact group \(\SU(2)\)
vanishes by averaging.
By \cite[Proposition~I.7 and
Remark~IV.16\textup{(3)}]{WW}, it follows that every locally
smooth \(\R/\Z\)-valued class on \(\Diff^+(S^2)\) has trivial pullback under
the standard rotation action of \(\SU(2)\).
On the other hand, 
the comparison above gives
\[
 L^*[\Phi_{\mathbb CP^1}]
 =
 N[c_P^{\rm CS}]
\]
with \(N\neq0\) \textup{(in the present normalization, \(N=8\))}.
Since \([c_P^{\rm CS}]\) has infinite order, no nonzero multiple
of \(\Phi_{\mathbb CP^1}\) can extend to \(G\).
\end{rem}

\appendix
\section{Proof of Proposition~\ref{43333}}\label{furoku}
We prove the proposition for the local filling from
Section~\ref{sec223344}.  Recall that this filling is defined using the
cutoff join
\[
 \mathrm{Join}_{\phi}(x,y;s)
 =x\,\phi^{-1}\bigl(\rho(s)\phi(x^{-1}y)\bigr),
 \qquad
 0\leq \rho\leq1,
 \quad \rho=0\ \text{near }0,
 \quad \rho=1\ \text{near }1.
\]
\begin{proof}[Proof of Proposition~\ref{43333}]

The local smoothness assertion follows immediately from the assumed
agreement with \(\mathcal F_\sigma(\tilde{\lambda})\).  After shrinking to a
connected neighborhood of \(e^{\times(i+1)}\), we may compute \(D\)
using the real-valued local lift
\[
 f(g_0,\ldots,g_i)
 :=
 \int_{\Delta^i}
 \sigma_i(g_0,\ldots,g_i)^*\tilde{\lambda}.
\]
By the assumed local identity, $cf$ is a real-valued local lift of
$\Phi$ near $e^{\times(i+1)}$.
Any two real local lifts differ by a locally constant
$\Lambda$-valued function, which is annihilated by $D$ because
$i\geq1$.  Hence, $
 D(\Phi)=D(cf)=cD(f).$
It remains to prove that
$D(f)=\tilde{\lambda}_e$.

Let \(\phi:U_0\to V_0\subset\mathfrak g\) be the chart used in the
construction, with
\(\phi(e)=0\) and \(d\phi_e=\operatorname{id}_{\mathfrak g}\).  In chart
coordinates put
\[
 J(a,b;s)
 :=
 \phi\!\left(
 \mathrm{Join}_{\phi}(\phi^{-1}(a),\phi^{-1}(b);s)
 \right).
\]
Since the differentials of inversion and multiplication at the identity
are \(-\operatorname{id}\) and addition, respectively,
\[
 d_{(a,b)}J_{(0,0;s)}(A,B)
 =
 (1-\rho(s))A+\rho(s)B.
\]
Consequently,
\[
 J(a,b;s)
 =
 (1-\rho(s))a+\rho(s)b+Q(a,b;s),
\]
where \(Q\), as well as \(\partial_s Q\), has vanishing first jet in
\((a,b)\) at \((0,0)\).  Thus, after composition with
finite-dimensional \(C^1\)-families of size \(O(|t|)\), \(Q=O(|t|^2)\).
Here \(|t|\) denotes the Euclidean norm on \(\mathbb R^i\), and $O$ is the Landau symbol.
Precisely all Landau estimates in the Fr\'echet chart are understood seminormwise:
for every continuous seminorm \(p\) on the model space and every fixed
finite-dimensional parameter family under consideration, there are
constants \(C_p\) and \(\varepsilon_p>0\) such that the stated estimate
holds for \(|t|<\varepsilon_p\).

Fix \(X_1,\ldots,X_i\in\mathfrak g\), and set
\[
 g_k(t)
 :=
 \exp(t_1X_1)\cdots\exp(t_kX_k),
 \qquad
 a_k(t):=\phi(g_k(t)).
\]
Then
\[
 a_k(t)
 =
 \sum_{j: 1 \leq j \leq k}t_jX_j+O(|t|^2).
\]
For \(0\leq n\leq i\), write
\[
 \beta_t^{(n)}
 :=
 \phi\circ
 \sigma_n(e,g_1(t),\ldots,g_n(t)),
 \qquad
 \beta_t^{(0)}=0.
\]

Let $\Psi_n(u,s)=((1-s)u,s)$.
For \(r\in[0,1]\), put
\(\rho_r(s)=(1-r)s+r\rho(s)\), and define recursively
\[
 R_{0,r}=\operatorname{id}_{\Delta^0},
 \qquad
 R_{n,r}(\Psi_n(u,s))
 =
 \Psi_n(R_{n-1,r}(u),\rho_r(s)).
\]
These are continuous maps of pairs
\((\Delta^n,\partial\Delta^n)\), with
\(R_{n,0}=\operatorname{id}_{\Delta^n}\).  
Here, for a simplex
\(\Delta^n\), this topology may be defined using any fixed finite atlas of manifolds with corners.
Set \(R_n:=R_{n,1}\).
In this sense, the cutoff conditions imply that \(R_n\) is smooth, and the preceding
homotopy shows that
\[
 (R_n)_*[\Delta^n,\partial\Delta^n]
 =
 [\Delta^n,\partial\Delta^n].
\]

For \(a=(a_1,\ldots,a_n)\), let
\[
 \ell_a(s_0,\ldots,s_n)
 :=
 \sum_{k: 1 \leq k \leq n}s_k a_k,
\]
and define
\[
 \alpha_t^{(n)}
 :=
 \ell_{(a_1(t),\ldots,a_n(t))}\circ R_n.
\]
Equivalently,
\[
 \alpha_t^{(n)}(\Psi_n(u,s))
 =
 (1-\rho(s))\alpha_t^{(n-1)}(u)+\rho(s)a_n(t).
\]
The recursive definition of \(\sigma_n\) and the first-jet formula for
\(J\) now give, by induction on \(n\),
\[
 \beta_t^{(n)}-\alpha_t^{(n)}
 =
 O(|t|^2)
 \quad\text{in }C^1(\Delta^n).
\]
Indeed, in cone coordinates the difference at level \(n\) is
\[
 (1-\rho(s))
 \bigl(\beta_t^{(n-1)}-\alpha_t^{(n-1)}\bigr)(u)
 +
 Q\bigl(\beta_t^{(n-1)}(u),a_n(t);s\bigr).
\]

Let
$
 \mu:= (\phi^{-1})^*\tilde{\lambda},
$ and $  \omega:= \mu_0$,%
where \(\omega\) is regarded as a constant \(i\)-form on
\(\mathfrak g\).  The maps \(\alpha_t^{(i)}\), \(\beta_t^{(i)}\), and
their first derivatives are \(O(|t|)\).  Since the smooth family
\(x\mapsto\mu_x-\omega\) vanishes at \(x=0\), the preceding
\(C^1\)-estimate and multilinearity imply
\[
 f(e,g_1(t),\ldots,g_i(t))
 =
 \int_{\Delta^i}
 (\alpha_t^{(i)})^*\omega
 +
 O(|t|^{i+1}).
\]

Now \(\alpha_t^{(i)}
=\ell_{(a_1(t),\ldots,a_i(t))}\circ R_i\), and \(R_i\) has relative
degree \(1\).  Hence
\[
 \int_{\Delta^i}(\alpha_t^{(i)})^*\omega
 =
 \int_{\Delta^i}
 \ell_{(a_1(t),\ldots,a_i(t))}^*\omega
 =
 \frac1{i!}
 \omega(a_1(t),\ldots,a_i(t)).
\]
By the expansion of the \(a_k(t)\) and the alternation of \(\omega\),
\[
 \omega(a_1(t),\ldots,a_i(t))
 =
 t_1\cdots t_i\,
 \omega(X_1,\ldots,X_i)
 +
 O(|t|^{i+1}).
\]
Therefore
\[
 f(e,g_1(t),\ldots,g_i(t))
 =
 \frac{ t_1\cdots t_i}{i!}
 \tilde{\lambda}_e(X_1,\ldots,X_i)
 +
 O(|t|^{i+1}).
\]
It follows that
$
 T^i f=\frac1{i!}\tilde{\lambda}_e.
$
Therefore, by the definition of $D$ and the alternation of
$\tilde{\lambda}_e$,
\[ D(f)(X_1,\ldots,X_i)
= \frac1{i!}
 \sum_{\pi\in\mathfrak S_i}
 \operatorname{sgn}(\pi)\,
 \tilde{\lambda}_e
 (X_{\pi(1)},\ldots,X_{\pi(i)})\\
 =
 \tilde{\lambda}_e(X_1,\ldots,X_i).\]
Combining this identity with
$D(\Phi)=cD(f)$ gives
$D(\Phi)=c\,\tilde{\lambda}_e$, as required.
\end{proof}

\small



\normalsize

\noindent
Department of Mathematics, Institute of Science Tokyo\\
2-12-1 Ookayama, Meguro-ku, Tokyo 152-8551, Japan

\end{document}